\documentclass[a4paper,reqno]{amsart}
\usepackage{amsmath,amssymb,amsfonts,amsthm,mathrsfs}
\usepackage{hyperref}
\usepackage{pdfsync}
\usepackage{color}

\newcommand{\C}{\mathbb C}

\newcommand{\R}{\mathbb R}
\newcommand{\N}{\mathbb N}

\def\({\left(}
\def\){\right)}
\def\<{\left\langle}
\def\>{\right\rangle}

\def\O{\mathcal O}

\def\G{\mathcal G}

\def\d{{\partial}}
\def\eps{\varepsilon}

\def\le{\leqslant}
\def\ge{\geqslant}

\DeclareMathOperator{\supp}{supp}

\DeclareMathOperator{\RE}{Re}
\DeclareMathOperator{\IM}{Im}
\DeclareMathOperator{\loc}{loc}
\DeclareMathOperator{\dist}{dist}
\DeclareMathOperator{\rad}{rad}

\def\Tend#1#2{\mathop{\longrightarrow}\limits_{#1\rightarrow#2}}

\theoremstyle{plain}
\newtheorem{theorem}{Theorem} [section]
\newtheorem{lemma}[theorem]{Lemma}

\newtheorem{proposition}[theorem]{Proposition}

\theoremstyle{remark}
\newtheorem{remark}[theorem]{Remark}

\theoremstyle{definition}

\newtheorem{observation}[theorem]{Observation}

\def\Tend#1#2{\mathop{\longrightarrow}\limits_{#1\rightarrow#2}}

\numberwithin{equation}{section}

\title[Rotational Binary BEC]
{On stability of rotational 2D binary Bose-Einstein condensates}

\author[R. Carles]{R\'emi Carles}
\address[R. Carles]{CNRS, IRMAR - UMR 6625, F-35000 RENNES, FRANCE}
\email{Remi.Carles@math.cnrs.fr}

\author[V. D. Dinh]{Van Duong Dinh}
\address[V. D. Dinh]{Laboratoire Paul Painlev\'e UMR 8524, Universit\'e de Lille CNRS, 59655 Villeneuve d'Asc, France
and
Department of Mathematics, HCMC University of Education, 280 An Duong Vuong, Ho Chi Minh, Vietnam}
\email{contact@duongdinh.com}

\author[H. Hajaiej]{Hichem Hajaiej}
\address[H. Hajaiej]{Department of Mathematics, California State University, Los Angeles, CA 90032}
\email{hhajaie@calstatela.edu}

\subjclass[2010]{35Q55, 35A01}
\keywords{Nonlinear Schr\"odinger equation; Bose-Einstein condensate; Harmonic potential; Rotation; Standing waves; Global existence; Stability}
\thanks{RC is supported by Rennes M\'etropole through its AIS
  program. VDD is supported by the Labex CEMPI (ANR-11-LABX-0007-01).}

\begin{document}

\begin{abstract}
  We consider a two-dimensional nonlinear Schr\"odinger equation
  proposed in Physics to model rotational binary Bose-Einstein condensates. The
  nonlinearity is a logarithmic modification of the usual cubic
  nonlinearity. The presence of both the external confining
  potential and rotating frame makes it difficult to apply standard
  techniques to directly construct ground states, as we explain in an
  appendix. The goal of the
  present paper is to analyze the orbital stability of the set of
  energy minimizers under mass constraint, according to the relative
  strength of the confining potential compared to the angular
  frequency. The main novelty concerns the critical case (lowest Landau Level) where these
  two effects compensate exactly, and orbital stability is
  established by using techniques related to magnetic Schr\"odinger
  operators.
\end{abstract}

\maketitle

\section{Introduction}\label{sec:intro}

The formation of self-bound droplets is a well-known macroscopic phenomenon. Recent experiments with ultracold quantum gases of bosonic atoms revealed a novel type of quantum liquid: dilute self-bound Bose-Einstein condensate (BEC) having orders of magnitude lower density than air  (see e.g. \cite{KSWWMFP, SWBFP, FKSWP, IMMHT} for Bose gases of dysprosium and \cite{SFMMWMMMIF, CTSNTCT} for binary Bose gases of potassium). Since these droplets form out of a BEC, there is good reason to assume that they have superfluid properties. One remarkable feature of a superfluid is its response to rotation, in particular the occurrence of quantized vortices (see \cite{Aftalion} for a broad introduction to these phenomena). In \cite{TSKR19}, binary BEC droplets carrying angular momentum were considered. Using weak first-order corrections to the mean field energy, often referred to as the Lee-Huang-Yang correction \cite{LeHuYa57}, a binary BEC droplet with angular momentum is well described by the wave function $\psi: \R_+ \times \R^2 \rightarrow \C$ whose evolution is governed by the Gross-Pitaevskii equation (GPE) with angular momentum
\begin{equation} \label{eq:nls}
i\d_t \psi +\frac{1}{2}\Delta \psi=V \psi+ |\psi|^2\ln (|\psi|^2) \psi -iK_3|\psi|^4\psi-\Omega L_z\psi,
\end{equation}
where the scaling invariances have been used to bring the equation into its dimensionless form. Here the external potential $V$ is of the form
\begin{equation} \label{eq:V-intro0}
V(x) = \frac{\gamma^2}{2} |x|^2 + V_0 e^{-\gamma|x|^2},
\end{equation}
where $\gamma > 0$ is the harmonic trap frequency and $V_0 \ge 0$ is the amplitude of the Gaussian. The parameter $K_3 \ge 0$ is the rate of three-body losses. The angular momentum operator $L_z$ is of the form
\begin{align} \label{eq:Lz-intro}
L_z = i (x_2 \partial_{x_1}-x_1 \partial_{x_2}), \quad x=(x_1,x_2) \in \R^2
\end{align}
and $\Omega> 0$ is the rotational speed.
The fact that
the constants in the harmonic trap and the Gaussian part of the
potential are equal stems from \cite{TSKR19}, but is not crucial in our
analysis, so we consider more generally
\begin{equation} \label{eq:V-intro}
V(x) = \frac{\gamma^2}{2} |x|^2 + V_0 e^{-\gamma_0|x|^2},\quad
  \gamma,\gamma_0>0.
\end{equation}
The main purpose of this paper is to study  the existence/nonexistence
and orbital stability of mass-constraint standing waves for
\eqref{eq:nls}.
We consider the Cauchy problem for \eqref{eq:nls} with initial data
$\psi_0 \in \Sigma$, where
\[
\Sigma:=\left\{f\in H^1(\R^2),\ x\mapsto |x|f(x)\in L^2(\R^2)\right\}
\]
is equipped with the norm
\[
\|f\|^2_\Sigma =\|f\|^2_{H^1} + \|xf\|^2_{L^2}.
\]
Due to the presence of the harmonic potential, this space is rather
natural (see e.g. \cite{Ca15}).
In the case $K_3=0$, there are three physical quantities which are formally conserved along the flow of \eqref{eq:nls}
\begin{align*}
M(\psi(t)) &= \|\psi(t)\|^2_{L^2} = M(\psi_0), \tag{Mass} \\
L(\psi(t)) &= \int_{\R^2} \bar{\psi}(t,x) L_z \psi(t,x) dx = L(\psi_0), \tag{Angular momentum} \\
E_\Omega(\psi(t)) &= \frac{1}{2}\|\nabla \psi(t)\|_{L^2}^2 + \int_{\R^2}V(x)
|\psi(t,x)|^2dx \tag{Energy}\\
&\mathrel{\phantom{=}}+
\frac{1}{2}\int_{\R^2}|\psi(t,x)|^4\ln\left(\frac{|\psi(t,x)|^2}{\sqrt e}\right)dx
-\Omega L(\psi(t)) = E_\Omega(\psi_0).
\end{align*}
Here we note that the angular momentum is real-valued, but has no definite sign. Similarly, the term involving the natural logarithm also has no definite sign.

In the case $K_3\ne 0$, solutions to \eqref{eq:nls} formally satisfy
\begin{equation}\label{eq:mass}
\frac{1}{2}\frac{d}{dt}\|\psi(t)\|_{L^2}^2+K_3\|\psi(t)\|_{L^6}^6=0.
\end{equation}
This shows that for $K_3 \not = 0$, the equation is irreversible. This
is the reason why, since $K_3\ge 0$, we
consider only positive time in the present paper. In the case $K_3 = 0$, the equation is reversible, and $\overline{\psi}(-t,x)$ solves \eqref{eq:nls}: considering the case $t \ge 0$ suffices to describes
the dynamics for all time.

When $\gamma=V_0=K_3=\Omega=0$, the
equation \eqref{eq:nls} was recently studied in \cite{CaSp-p}. More
precisely, the global well-posedness for $H^1$ data, the existence and
uniqueness of positive ground state solutions for \eqref{eq:nls} were
shown there, along with  the orbital stability of prescribed mass
standing waves.

Our first result is the following global existence for \eqref{eq:nls}.

\begin{theorem}[Global well-posedness]\label{theo:GWP}
Let $\gamma,\gamma_0>0$, $\Omega > 0$, $V_0\ge 0$, $K_3\ge 0$, and $\psi_0 \in  \Sigma$. Then there exists a unique global-in-time solution to \eqref{eq:nls} satisfying $\psi\in C(\R_+; \Sigma)\cap L^3_{\loc}(\R_+;L^6(\R^2))$.
\begin{itemize}
\item If $K_3=0$, the conservation laws of mass, angular momentum, and energy hold.
\item If $K_3>0$, the solution asymptotically vanishes in the sense that
\begin{equation} \label{eq:decay}
  \|\psi(t)\|^2_{L^2} =\O\(t^{-1/4}\)\quad\text{as }t\to \infty.
\end{equation}
\end{itemize}
\end{theorem}
The Cauchy problem is addressed by resuming the approach from
\cite{AnMaSp12}. In passing, we fix a small flaw present in this paper
regarding dispersive estimates. The asymptotic extinction
\eqref{eq:decay} is then established like in \cite{AnCaSp15}, thanks
to a suitable uniform bound which makes it possible to control the
$L^6$-norm in \eqref{eq:mass} from below in terms of the $L^2$-norm.
\begin{remark}
The potential $V$ in \eqref{eq:V-intro} is radially symmetric, since
it is the model given in \cite{TSKR19}. We will see in the proof of
Theorem~\ref{theo:GWP} that the result still holds true in the more general
case of smooth potentials which are at most quadratic (and thus need
not be radial); see Remarks~\ref{rem:nonradial1} and \ref{rem:nonradial2}.
\end{remark}
\smallbreak

In the rest of the introduction, we are interested in the absence of three-body losses, i.e. $K_3=0$. In this case, \eqref{eq:nls} admits standing waves, i.e. solutions of the form
\begin{align} \label{eq:stan-wave}
\psi(t,x) = e^{i\omega t} \phi(x), \quad \omega \in \R,
\end{align}
where $\phi$ solves
\begin{equation}
 \label{eq:ground}
 -\frac{1}{2}\Delta \phi +V\phi+\phi|\phi|^2\ln (|\phi|^2) -\Omega L_z\phi+\omega \phi=0,\quad x\in \R^2.
\end{equation}
Note that in the case $K_3>0$, there is no such solution in view of the asymptotic extinction \eqref{eq:decay}.

The existence of standing waves for \eqref{eq:ground} can be achieved by several ways. The first way is to minimize the energy functional
\[
E_\Omega(f) = \frac{1}{2} \|\nabla f\|^2_{L^2} + \int_{\R^2} V |f|^2 dx + \frac{1}{2} |f|^4 \ln \left( \frac{|f|^2}{\sqrt{e}}\right) dx - \Omega L(f),
\]
with prescribed mass constraint, i.e. $\|f\|^2_{L^2} = \rho>0$, an
strategy which is often adopted in Physics. In this case, the parameter $\omega$ in \eqref{eq:ground} appears as a Lagrange multiplier associated to the minimization problem. Another way is to look for critical points of the action functional
\begin{align*}
S_\omega(f) &= E_\Omega(f) + \omega M(f) \\
&= \frac{1}{2} \|\nabla f\|^2_{L^2} + \omega \|f\|^2_{L^2} +
   \int_{\R^2} V |f|^2 dx + \frac{1}{2} \int_{\R^2}
       |f|^4 \ln \left( \frac{|f|^2}{\sqrt{e}}\right)   dx - \Omega L(f),
\end{align*}
with $\omega$ being given and fixed. However, this approach seems
difficult to apply in the present context.  More precisely,
\eqref{eq:ground} has two features which make it difficult to
characterize the range of $\omega$'s allowed to find a non-trivial
solution to \eqref{eq:ground}. The presence of the external potential
$V$ and the rotation $L_z$ introduces an $x$-dependence which makes it
impossible to invoke the results from \cite{BeGaKa83} (the 2D
counterpart of \cite{BeLi83b}), or even adapt easily the proof. On the
other hand, the fact that the nonlinearity is not homogeneous in
$\phi$ makes it impossible to reproduce the arguments from
\cite{RoWe88} (see also \cite{Fu01} in the case of a harmonic
potential). In an appendix, we collect some information regarding the
possible range for $\omega$ in the radial case, where the rotating
term is absent from \eqref{eq:ground}, and explain in more details why
minimizing the action seems difficult here.
\smallbreak

We therefore consider the following minimization problem: for $\rho>0$,
\begin{align} \label{I-Omega-rho}
I_\Omega(\rho) := \inf \left\{ E_\Omega(f) \ : \ f \in \Sigma, \|f\|^2_{L^2} = \rho\right\}.
\end{align}
Our next result concerns the existence and stability of prescribed mass standing waves for \eqref{eq:ground} in the case of low rotational speed.
\begin{theorem} \label{theo:orbi-low}
	Let $K_3=0$, $\gamma,\gamma_0>0, V_0\ge 0$ and $0<\Omega<\gamma$. Then for any $\rho>0$, there exists $\phi \in \Sigma$ such that $E_\Omega(\phi) = I_\Omega(\rho)$ and $\|\phi\|^2_{L^2} =\rho$. Moreover, the set
	\[
	\G_\Omega(\rho):= \left\{ \phi \in \Sigma \ : \ E_\Omega(\phi) = I_\Omega(\rho), \|\phi\|^2_{L^2} =\rho\right\}
	\]
	is orbitally stable under the flow of \eqref{eq:nls} in the sense that for any $\epsilon>0$, there exists $\delta>0$ such that for any initial data $u_0 \in \Sigma$ satisfying
	\[
	\inf_{\phi \in \G_\Omega(\rho)} \|u_0 - \phi\|_{\Sigma} <\delta,
	\]
	the corresponding solution to \eqref{eq:nls} exists globally in time and satisfies
	\[
\sup_{t\in \R}	\inf_{\phi \in \G_\Omega(\rho)} \|u(t)-\phi\|_{\Sigma}
<\epsilon .
	\]
\end{theorem}

The proof of Theorem \ref{theo:orbi-low} is based on a standard variational argument using the following observation (see Lemma \ref{lem:norm}): for $0<\Omega<\gamma$,
\[
\|\nabla f\|^2_{L^2} + 2\int_{\R^2} V|f|^2 dx -2\Omega L(f) \simeq \|\nabla f\|^2_{L^2} + \|xf\|^2_{L^2}
\]
which enables us to use the standard compact embedding $\Sigma \hookrightarrow L^r(\R^2)$ for all $2\le r<\infty$. We also make use of the log-type inequality
\begin{align} \label{eq:est-log}
\left| \int_{\R^2} |f|^4 \ln \left(\frac{|f|^2}{\sqrt{e}}\right) dx\right| \lesssim_\epsilon \|f\|^{4-\epsilon}_{L^{4-\epsilon}} + \|f\|^{4+\epsilon}_{L^{4+\epsilon}}
\end{align}
for any $\epsilon>0$. For more details, we refer to Subsection 3.1.

Next we consider the critical rotational speed $\Omega =\gamma$
(lowest Landau level, see e.g. \cite{Aftalion} and
references therein),
which
constitutes the main novelty of this paper. In this case, the energy functional can be rewritten as
\begin{align} \label{eq:E-Omega}
E_\gamma(f) = \frac{1}{2}\|\nabla_A f\|^2_{L^2} + V_0 \int_{\R^2} e^{-\gamma_0 |x|^2} |f(x)|^2 dx+ \frac{1}{2} \int_{\R^2} |f|^4 \ln\left(\frac{|f|^2}{\sqrt{e}}\right) dx,
\end{align}
where $\nabla_A := \nabla - iA$ with $A(x) = \gamma(-x_2, x_1)$. Thanks to this observation, it is convenient to consider the above functional on the magnetic Sobolev space
\begin{align} \label{eq:H1A}
H^1_A (\R^2):= \left\{ f \in L^2(\R^2) \ : \ (\partial_j - i A_j) f \in L^2(\R^2), j=1,2 \right\}
\end{align}
endowed with the norm
\[
\|f\|^2_{H^1_A} := \|f\|^2_{L^2} + \|\nabla_A f\|^2_{L^2}.
\]
Note that the energy functional \eqref{eq:E-Omega} is well-defined on $H^1_A(\R^2)$ by using \eqref{eq:est-log} and the magnetic Gagliardo-Nirenberg inequality (see e.g., \cite{EsLi89}): for $2<r<\infty$,
\[
\|f\|^{r}_{L^r} \le C_r \|\nabla_A f\|^{r-2}_{L^2} \|f\|^2_{L^2}, \quad \forall f \in H^1_A(\R^2).
\]
We also have $\Sigma \subset H^1_A(\R^2)$. Indeed, for $f \in \Sigma$, we have
\begin{align*}
\|f\|^2_{H^1_A} = \|f\|^2_{L^2} + \|\nabla f\|^2_{L^2} + \gamma^2 \|xf\|^2_{L^2} - \gamma \int_{\R^2} \bar{f} L_z f dx  \le C \|f\|^2_{\Sigma},
\end{align*}
where the last inequality follows from
\[
|L(f)| \le \|xf\|_{L^2} \|\nabla f\|_{L^2} \le
\frac{1}{2\gamma}\|\nabla f\|^2_{L^2} + \frac{\gamma}{2} \|xf\|^2_{L^2}.
\]

To show the existence and stability of standing waves to
\eqref{eq:nls} in the critical rotational case, it is convenient to consider separately two cases: $V_0 =0$ and $V_0>0$.

When $V_0 =0$, we denote
\begin{align} \label{eq:I-Omega}
I^0_\gamma(\rho):= \inf \left\{ E^0_\gamma(f) \ : \ f \in H^1_A(\R^2), \|f\|^2_{L^2} =\rho\right\},
\end{align}
where
\begin{align} \label{eq:E0}
E^0_\gamma(f):= \frac{1}{2}\|\nabla_A f\|^2_{L^2} + \frac{1}{2} \int_{\R^2} |f|^4 \ln\left(\frac{|f|^2}{\sqrt{e}}\right) dx,
\end{align}
and the superscript $0$ is here to emphasize the assumption $V_0=0$.
Let us first introduce the cubic ground state, that is, the unique positive,
radially symmetric solution to
\begin{equation}
  \label{eq:Q}
  -\frac{1}{2}\Delta Q + Q = Q^3,\quad x\in \R^2.
\end{equation}
By making use of a variant of the celebrated concentration-compactness principle of Lions adapted to the magnetic Sobolev space $H^1_A(\R^2)$ (see Lemma \ref{lem:conc-comp}), we prove the following result.

\begin{theorem} \label{theo:orbi-cri-1}
	Let $K_3=V_0=0$, $0<\gamma<\tfrac{1}{ 2 e^{3/2}}$ and $\Omega =\gamma$. Let $0<\rho \le \|Q\|_{L^2}^2$. Then there exists $\phi \in H^1_A(\R^2)$ such that $E^0_\gamma (\phi) = I^0_\gamma(\rho)$ and $\|\phi\|^2_{L^2} =\rho$. Moreover, the set of minimizers for $I^0_\gamma(\rho)$ is orbitally stable under the flow of \eqref{eq:nls}.
\end{theorem}

The assumption $\rho \le \|Q\|^2_{L^2}$ is probably technical, due to
our argument (see after \eqref{eq:cond-rho}). The proof also relies on
the property $I^0_\gamma(\rho) <0$ (see Proposition~\ref{prop:orbi-cri-1}). In the case $K_3=V_0=\gamma=\Omega=0$, the latter condition was shown in \cite{CaSp-p} for any $\rho>0$. In our setting, showing that $I^0_\gamma(\rho)<0$ is more complicated since the scaling argument used in \cite{CaSp-p} does not work because of the presence of magnetic potential. However, by using a trial function of the form $\lambda e^{-b|x|^2}$ with suitable positive constants $\lambda$ and $b$, we are able to show (see Lemma~\ref{lem:data}) that both conditions $\rho \le \|Q\|^2_{L^2}$ and $I^0_\gamma(\rho) <0$ are fulfilled by some data $f \in H^1_A(\R^2)$. This also leads to a restriction  on the validity of $\gamma$. We refer the reader to Subsection 3.2 for more details.

When $V_0>0$, the above-mentioned concentration-compactness argument does not work due to the lack of spatial translation of the term
\[
\int_{\R^2} e^{-\gamma_0|x|^2} |f(x)|^2 dx.
\]
To overcome the difficulty, we restrict our consideration on
$H^1_{A,\rad}(\R^2)$ the space of radially symmetric functions of
$H^1_A(\R^2)$. Note that this restriction has a drawback since we no longer see the effect of rotation to the equation as $L_z f=0$ for a radial function $f$. We consider
\begin{align} \label{eq:I-rho-rad}
I_{\gamma,\rad}(\rho):= \inf \left\{E_\gamma(f) \ : \ f \in H^1_{A,\rad}(\R^2), \|f\|^2_{L^2} = \rho\right\},
\end{align}
where $E_\gamma(f)$ is as in \eqref{eq:E-Omega}. We have the following result.
\begin{theorem} \label{theo:orbi-cri-2}
	Let $K_3=0$, $V_0 \ge 0$, $\gamma,\gamma_0>0$, and $\Omega =\gamma$. Then for any $\rho>0$, there exists $\phi \in H^1_{A,\rad}(\R^2)$ such that $E_\gamma(\phi) = I_{\gamma,\rad}(\rho)$ and $\|\phi\|^2_{L^2} = \rho$. Moreover, the set of minimizers for $I_{\gamma,\rad}(\rho)$ is orbitally stable under the flow of \eqref{eq:nls}.
\end{theorem}

\begin{remark}
	As mentioned above that the rotation term vanishes for radially symmetric functions, the result in Theorem \ref{theo:orbi-cri-2} holds for any $\Omega \geq 0$. In particular, it gives another (radial) solution, besides the one obtained in Theorem \ref{theo:orbi-cri-1}, to \eqref{eq:ground}.
\end{remark}
The proof of Theorem \ref{theo:orbi-cri-2} relies on the following compact embedding
\begin{align} \label{eq:com-emb}
H^1_{A,\rad}(\R^2) \ni f \mapsto |f| \in L^r(\R^2)
\end{align}
for all $2 \le  r<\infty$. This compact embedding follows from the well-known compact embedding $H^1_{\rad}(\R^2)\hookrightarrow L^r(\R^2)$ for all $2 \le r<\infty$ and the fact that
\[
H^1_{A,\rad}(\R^2) \ni f \mapsto |f| \in H^1_{\rad}(\R^2)
\]
is continuous due to the diamagnetic inequality (see e.g. \cite{LiLo01})
\[
|\nabla |f|(x)| \le |\nabla_A f(x)| \quad \text{a.e. } x \in \R^2.
\]

When the rotational speed exceeds the critical value,
i.e. $\Omega>\gamma$, we have the following nonexistence
of minimizers for the constrained variational problem
  $I_\Omega(\rho)$. 

\begin{theorem} \label{theo:non-exis}
	Let $K_3 =0$, $\gamma,\gamma_0>0$, $V_0\ge 0$, and $\Omega
        >\gamma$. Then for any $\rho>0$, there is no minimizer for
        $I_\Omega(\rho)$,  i.e., $I_\Omega(\rho)=-\infty$.
\end{theorem}

The proof of Theorem \ref{theo:non-exis} is based on an idea of Bao, Wang, and Markowich
\cite[Section~3.2]{BWM}, using the central vortex state with winding number $m$, namely
\[
f_m(x)=f_m(r,\theta) =\sqrt{\frac{\rho \gamma^{m+1}}{\pi m!}}r^m e^{-\frac{\gamma|x|^2}{2}} e^{im\theta}, \quad m \in \N,
\]
where $(r,\theta)$ are the polar coordinates in $\R^2$. Physically, when the angular velocity of rotation exceeds the trapping frequency,
the harmonic potential cannot provide enough necessary centripetal force that counteracts the centrifugal force caused by the rotation,
and the gas may fly apart. In the classical rotating BEC with purely power-type nonlinearity, the nonexistence of prescribed mass standing
waves was proved by Bao, Wang, and Markowich \cite{BWM}.

This work is organized as follows. In Section \ref{sec:cauchy}, we
study the local well-posedness for \eqref{eq:nls} and prove the global
existence given in Theorem \ref{theo:GWP}. In Section \ref{sec:orbi},
we investigate the existence/nonexistence and the orbital stability of
prescribed mass standing waves for \eqref{eq:nls} given in Theorems \ref{theo:orbi-low}, \ref{theo:orbi-cri-1}, \ref{theo:orbi-cri-2}, and
\ref{theo:non-exis}. Finally, we give some information on stationary
solutions of \eqref{eq:ground} in the radial case in
Appendix~\ref{sec:ground}, and characterize prescribed mass minimizers
in Appendix~\ref{sec:charac}.

\section{Cauchy problem}\label{sec:cauchy}
In the case $K_3 = 0$, Theorem \ref{theo:GWP} is reminiscent of the results from \cite{AnMaSp12} for the
local and global well-posedness. We present an alternative proof, to show that local solutions are actually global, which simplifies the approach, and yields better bounds. The asymptotic extinction in the case $K_3>0$ is obtained by adapting the arguments from \cite{AnSp10} based on the introduction of a suitable pseudo-energy (see also \cite{AnCaSp15} for generalizations).

\subsection{Local well-posedness}
\label{sec:LWP}

Local well-posedness follows from a fixed point argument on Duhamel's formula
\begin{equation}\label{eq:duhamel}
\psi(t) =U(t) \psi_0 - i \lambda \int_0^t U(t-s)f(u)(s) \, ds,
\end{equation}
where here and in the following, we denote
\[
f(z)= z|z|^2\ln (|z|^2)-iK_3|z|^4z, \quad z\in \C.
\]
The notation $U(\cdot)$ stands for the linear propagator, that is, $U(t) \psi_0 =\psi_{\rm lin}(t)$ is the solution to
\begin{equation}\label{eq:lin}
 i\d_t \psi_{\rm lin} +\frac{1}{2}\Delta \psi_{\rm lin}
 =V(x) \psi_{\rm lin}
 -\Omega L_z\psi_{\rm lin} ,\quad  \psi_{\rm lin}(0,x) = \psi_0.
\end{equation}
As noticed in \cite{AnMaSp12}, $U(t) = e^{-itH(x,D_x)}$, where
\begin{equation*}
  H(x,\xi) = \frac{1}{2}|\xi|^2 +V(x) +\Omega (x_1\xi_2-x_2\xi_1),
\end{equation*}
enters into the general framework of \cite{Kitada80}, where the fundamental solution, that is, the kernel associated to $U(\cdot)$, is constructed. However, this does not yield directly local in time dispersive properties, since unlike in \cite{Fujiwara}, the solution to\footnote{We do not include here the semi-classical parameter present in \cite{Fujiwara,Kitada80}.}
\begin{equation*}
  i\d_t \psi  + H(t,x,-i\d_x)\psi=0\quad ;\quad \psi_{\mid t=0}=\varphi,
\end{equation*}
for $H=H(t,x,\xi)$ smooth, real-valued and at most quadratic in $(x,\xi)$ (locally in time), is represented as
\begin{equation*}
  U(t)\varphi(x) = \int_{\R^2}e^{i\phi(t,x,\xi)}a(t,x,\xi)\hat \varphi(\xi)d\xi,
\end{equation*}
that is, an oscillatory integral involving the Fourier transform of $\varphi$, and not $\varphi$ directly. Consequently, dispersive $L^1-L^\infty$ estimates cannot be inferred in general, as shown by the trivial case $H=0$: it enters the framework of \cite{Kitada80}, the above representation holds with $a=\text{constant}$ (whose value depends on the definition of the Fourier transform), but no Strichartz estimate is available, of course.
\smallbreak

In the present case, local dispersive estimates are available though, as we show by relying on the isotropy of $V$ and $L_z$; this corresponds to the trick presented in
\cite{AnMaSp12} in the 3D case. This alternative approach not only yields dispersive estimates, but also makes it possible to study the Cauchy problem as if the rotation term was absent. Define the function $\varphi$ by
\begin{equation}\label{eq:chgt-fct}
  \varphi(t,x) = \psi\(t,x_1\cos(\Omega t)+x_2\sin(\Omega t),-x_1\sin(\Omega
  t)+x_2\cos (\Omega t)\).
\end{equation}
We check that since $V$ depends only on $|x|^2$ (and is therefore invariant under rotation), $\psi$ solves \eqref{eq:nls} if and only if $\varphi$
solves
\begin{equation}\label{eq:nls2}
\left\{
\begin{aligned}
 i\d_t \varphi +\frac{1}{2}\Delta \varphi
 &=V \varphi+
 \varphi|\varphi|^2\ln (|\varphi|^2) -iK_3|\varphi|^4\varphi,\quad (t,x)\in \R_+ \times \R^2,\\
  \varphi(0,x) &= \psi_0,
\end{aligned}
\right.
\end{equation}
that is, \eqref{eq:nls} with $\Omega=0$. Now for the perturbed harmonic oscillator $H = -\tfrac{1}{2}\Delta+V$, local in time dispersive estimates are available: from \cite{Fujiwara}, there exists $\delta>0$ such that
\begin{equation}\label{eq:disp-free}
  \|e^{-itH}\psi_0\|_{L^\infty}\lesssim
  \frac{1}{|t|}\|\psi_0\|_{L^1},\quad |t|\le \delta.
\end{equation}
Note that $\delta$ is necessarily finite, as $H$ possesses eigenvalues (as a consequence from e.g., \cite[Theorem XIII.67]{ReedSimon4}). This implies local in time Strichartz estimates (see e.g., \cite{CazCourant}) and so, as long as bounded time intervals only are involved, \eqref{eq:duhamel} can be studied like in the case where $U(t)= e^{it\Delta}$. Note that since \eqref{eq:chgt-fct} preserves the Lebesgue norms, we infer
\begin{equation}\label{eq:disp-omega}
\|e^{-it H_\Omega} \psi_0\|_{L^\infty} \lesssim \frac{1}{|t|} \|\psi_0\|_{L^1}, \quad |t| \le \delta,
\end{equation}
for all $\Omega \in \R$, with $\delta$ and an implicit multiplicative constant independent of $\Omega$.

In the case $K_3=0$, the analysis meets essentially the one
presented in \cite{CaSp-p}. In the general case $K_3\ge 0$,   $f\in C^1(\R^2;\R^2)$ satisfies $f(0)=0$,
\[
| f(u)| \lesssim |u|^{3-\eps} + |u|^{4} , \quad \forall \eps>0,
\]
as well as
\[
|\nabla f(u)|\lesssim ( |u|^{2-\eps}+|u|^3) |\nabla u|.
\]
We then obtain:
\begin{lemma}[Local well-posedness]\label{lem:lwp}
Let $\gamma>0, \Omega>0$, $V_0 \ge 0, K_3 \ge 0$, and $\psi_0 \in \Sigma$. Then there exist $T>0$ and a unique solution
\[
\varphi\in C([0, T];\Sigma)\cap L^3((0, T); L^6(\R^2)),
\]
to \eqref{eq:nls2}. Equivalently, in view of \eqref{eq:chgt-fct},
\[
\psi\in C([0, T];\Sigma)\cap L^3((0, T); L^6(\R^2)),
\]
solves \eqref{eq:nls}. Moreover, it satisfies \eqref{eq:mass}.  Either the solution is global in time,
\[
\varphi,\psi\in C(\R_+;\Sigma)\cap L^3_{\rm loc}(\R_+; L^6(\R^2)),
\]
or there exists  $T^*<\infty$ such that
\[
\lim_{t\to T^*} \| \nabla \varphi(t)\|_{L^2}=\lim_{t\to T^*} \| \nabla
\psi(t)\|_{L^2}=\infty.
\]
In the case $K_3=0$, we also have the conservation laws of mass, angular momentum, and energy.
\end{lemma}

\begin{proof}[Sketch of the proof]
We present the elements of the proof which require a little modification due to the presence of the damping term ($K_3>0$).

We recall that the standard blow-up alternative involves the condition
\[
\lim_{t\to T^*} \| \varphi(t)\|_{\Sigma}=\lim_{t\to T^*} \| \psi(t)\|_{\Sigma}=\infty.
\]
In view of \eqref{eq:mass}, the $L^2$-norm of $\psi$ remains bounded for positive time, and thus, so does the $L^2$-norm of $\varphi$. On the other hand, we compute
\begin{align*}
  \frac{d}{dt}\|x\varphi(t)\|_{L^2}^2& = -\IM
   \int_{\R^2}|x|^2\bar\varphi\Delta\varphi dx
                                       -2K_3\int_{\R^2}|x|^2|\varphi|^6
                                       dx\\
                                     & = 2\IM \int_{\R^2}\bar\varphi x\cdot\nabla \varphi dx
   -2K_3\int_{\R^2}|x|^2|\varphi|^6 dx.
\end{align*}
Using Cauchy--Schwarz and Young's inequalities, we infer that
\begin{equation}\label{eq:virial}
   \frac{d}{dt}\|x\varphi(t)\|_{L^2}^2\le \|x\varphi(t)\|_{L^2}^2 +\|\nabla
  \varphi(t)\|_{L^2}^2 .
\end{equation}
Therefore, if $\nabla \varphi$ remains bounded in $L^2$, then so does the $\Sigma$-norm, and the blow-up criterion can be reduced to the one stated in the lemma.

In the case $K_3=0$, the proof of conservation laws follows, for instance, from \cite{Oz06}.
\end{proof}

\begin{remark}[Non-radial potential]\label{rem:nonradial1}
  The case of a non-radial potential can be treated along essentially
  the same lines. Indeed, for a potential $V\in C^\infty(\R^2;\R)$,
  which is at most quadratic in the sense of \cite{Fujiwara}, that is,
  \begin{equation*}
    \d_x^\alpha V\in L^\infty(\R^2),\quad \forall \alpha\in \N^2, \
    |\alpha|\ge 2,
  \end{equation*}
  the function $\varphi$, given by \eqref{eq:chgt-fct}, solves
  \begin{equation*}
    i\d_t \varphi +\frac{1}{2}\Delta \varphi
 =\tilde V(t,x) \varphi+
 \varphi|\varphi|^2\ln (|\varphi|^2) -iK_3|\varphi|^4\varphi,
\end{equation*}
with $\tilde V(t,x)= V(x_1\cos(\Omega t)+x_2\sin(\Omega t),-x_1\sin(\Omega
  t)+x_2\cos (\Omega t))$. The time dependent potential $\tilde V$ is
  smooth in $(t,x)$, and at most quadratic in space (in the same sense
  as above), uniformly in time. It follows from \cite{Fujiwara} that
  \eqref{eq:disp-free} remains valid with $e^{-itH}$ replaced by the
  evolution group $U(s+t,s)$ associated to $\tilde H(t)
  =-\frac{1}{2}\Delta+\tilde V(t,x)$, hence \eqref{eq:disp-omega}
  holds in this case too. The rest of the proof of Lemma~\ref{lem:lwp}
  can be repeated exactly then.
\end{remark}

\subsection{Global well-posedness}
\label{sec:GWP}

In the case $K_3=0$, the following energy is independent of time,
\begin{equation*}
   E_0(\varphi) := \frac{1}{2}\|\nabla \varphi\|_{L^2}^2 + \int_{\R^2}V
  |\varphi|^2 dx+
  \frac{1}{2}\int_{\R^2}|\varphi|^4\ln\left(\frac{|\varphi|^2}{\sqrt e}\right)dx,
\end{equation*}
as well as the mass of $\varphi$. The positive part of the energy satisfies
\begin{align*}
  E_+(\varphi(t)) :&= \frac{1}{2}\|\nabla \varphi(t)\|_{L^2}^2
            +\int_{\R^2}V|\varphi(t)|^2 dx+ \frac{1}{2}
           \int_{|\varphi|^2>\sqrt
                     e}|\varphi(t,x)|^4\ln\(\frac{|\varphi(t,x)|^2}{\sqrt
                     e}\)dx\\
  &=
  E_0(\psi_0)+ \frac{1}{2}
    \int_{|\varphi|^2<\sqrt e}|\varphi(t,x)|^4\ln\(\frac{\sqrt
    e}{|\psi(t,x)|^2}\) dx\\
&\le  E_0 (\psi_0)+   \frac{\sqrt e}{2}  \int_{\R^2}|\varphi(t,x)|^2 dx,
\end{align*}
where we have used the easy bound $\ln a \le a$ for $a\ge 1$. Using the conservation of mass, this yields
\begin{equation*}
   E_+(\varphi(t))\lesssim 1,
\end{equation*}
and thus global existence (in the past as well, since for $K_3=0$, the equation is reversible), since $\|\varphi\|_{\Sigma}\lesssim E_+(\varphi)$. Note that even the $L^2$-norm of $\varphi$ is controlled by $E_+$, in view of the uncertainty principle (see e.g., \cite[\S~3.2]{Rauch91})
\begin{equation} \label{eq:uncer}
\|f\|^2_{L^2} \le \|\nabla f\|_{L^2} \|xf\|_{L^2}, \quad f \in \Sigma.
\end{equation}
 In particular, there exists $C>0$ such that
\begin{equation}\label{eq:borne-moment}
  \int_{\R^2} |x|^2|\varphi(t,x)|^2dx =  \int_{\R^2}
  |y|^2|\psi(t,y)|^2dy\le C,\quad \forall t\ge 0.
\end{equation}
\smallbreak

In the case $K_3>0$, following the strategy presented in \cite[Section~3.1]{AnSp10}, we introduce a pseudo-energy, for $k>0$,
\begin{equation*}
  E(k,\varphi) =E_0(\varphi)+k\|\varphi\|_{L^6}^6.
\end{equation*}
Adapting slightly the computations from \cite{AnCaSp15}, we find:
\begin{align*}
  \frac{d}{dt}E(k,\varphi(t))  &=K_3 \int_{\R^2}
     |\varphi|^4\RE\(\bar \varphi\Delta\varphi\)dx -3k\int_{\R^2}
  |\varphi|^4\IM\(\bar\varphi\Delta\varphi\)dx\\
             &\quad            -2K_3 \int_{\R^2}  V |\varphi|^6dx
 -2K_3\int_{\R^2} |\varphi|^8\ln(|\varphi|^2)dx-6kK_3\int_{\R^2} |\varphi|^{10}dx.
\end{align*}
Again, following the computations from \cite[Section~3.1]{AnSp10}, we introduce the polar factor related to $\varphi$,
\begin{equation*}
  \theta(t,x):=
\left\{
  \begin{array}{cl}
   \tfrac{\varphi(t,x)}{ |\varphi(t,x)|} & \quad\text{ if }\varphi(t,x)\not =0,\\
0& \quad\text{ if }\varphi(t,x) =0,
  \end{array}
\right.
\end{equation*}
and the above time derivative can be rewritten as
\begin{align*}
  \frac{d}{dt}E(k,\varphi(t))  &=-(K_3-6k) \int_{\R^2}
|\varphi|^4|\nabla\varphi|^2dx -4K_3\int_{\R^2}
|\varphi|^4\left|\nabla|\varphi|\right|^2dx \\
 &\quad                       -6k\int_{\R^2}
  |\varphi|^4\left| \RE(\bar\theta \nabla \varphi)
   -\IM(\bar\theta\nabla\varphi)\right|^2 dx
                        -2K_3 \int_{\R^2}  V |\varphi|^6dx \\
  &\quad-2K_3\int_{\R^2} |\varphi|^8\ln(|\varphi|^2)dx-6kK_3\int_{\R^2} |\varphi|^{10}dx.
\end{align*}
Picking $0<k<K_3/6$, the above relation implies
\begin{equation*}
   \frac{d}{dt}E\(k,\varphi(t)\)\le 2K_3\int_{|\varphi|<1}|\varphi|^8\ln
 \frac{1}{ |\varphi|^2} dx \le 2K_3\int_{\R^2} |\varphi|^6 dx.
\end{equation*}
Since the change of unknown \eqref{eq:chgt-fct} preserves the Lebesgue norms, \eqref{eq:mass} yields
\begin{equation*}
2K_3\int_{\R^2} |\varphi|^6 = -\frac{d}{dt}\|\varphi\|_{L^2}^2,
\end{equation*}
hence
\begin{equation*}
   \frac{d}{dt}\(E\(k,\varphi(t)\)+\|\varphi\|_{L^2}^2\)\le 0,
 \end{equation*}
 and we conclude like in the case $K_3=0$. In particular,
 \eqref{eq:borne-moment} holds.

 \begin{remark}[Non-radial potential, continued]\label{rem:nonradial2}
  For a
   general potential $V$ like in Remark~\ref{rem:nonradial1}, the natural energy associated to
   $\varphi$ becomes
   \begin{equation*}
      E_0(\varphi(t))=  \frac{1}{2}\|\nabla \varphi(t)\|_{L^2}^2 +
     \int_{\R^2}\tilde V(t,x)
  |\varphi(t,x)|^2 dx+
  \frac{1}{2}\int_{\R^2}|\varphi(t,x)|^4\ln\left(\frac{|\varphi(t,x)|^2}{\sqrt e}\right)dx,
\end{equation*}
and even for $K_3=0$, it is not constant in time if $V$ is not radial
($\d_t \tilde V\not
\equiv 0$). However, following either the virial computation from
\cite{Ca11} or the pseudo-energy argument from \cite{CaSi15}, we can
easily adapt the above argument and infer
that the global well-posedness in Theorem~\ref{theo:GWP} remains valid.
 \end{remark}
\subsection{Large time extinction in the presence of loss}
\label{sec:loss}

In view of \cite[Lemma~4.1]{AnCaSp15},
\begin{equation*}
  \|\varphi\|_{L^2} \lesssim \|\varphi\|_{L^6}^{3/5}\|x
  \varphi\|_{L^2}^{2/5} \lesssim \|\varphi\|_{L^6}^{3/5},
\end{equation*}
where we have used \eqref{eq:borne-moment}. This, together with \eqref{eq:mass}, yields
\begin{equation*}
  \frac{d}{dt}\|\varphi(t)\|_{L^2}^2 + C \|\varphi(t)\|_{L^2}^{10}\le 0,
\end{equation*}
for some uniform $C>0$. Recalling that the (nonnegative) solution of the ODE $\dot y +Cy^5=0$ with $y(0)=\|\psi_0\|_{L^2}^2$
satisfies $y(t)=\O(t^{-1/4})$,
this implies
\begin{equation*}
  \|\varphi(t)\|_{L^2}^2 \lesssim \frac{1}{t^{1/4}},
\end{equation*}
hence the decay announced in Theorem~\ref{theo:GWP}, since
$\|\varphi(t)\|_{L^2}= \|\psi(t)\|_{L^2}$.
Again, this conclusion remains true when $V$ satisfies the assumptions of
Remark~\ref{rem:nonradial1} and is not necessarily radial.

\section{Orbital stability} \label{sec:orbi}
In this section, we study the existence/nonexistence and stability of constraint mass standing waves associated to \eqref{eq:ground}. Recall that $K_3=0$ throughout this section. We will consider separately three cases: low rotational speed ($0<\Omega<\gamma$), critical rotational speed ($\Omega =\gamma$), and high rotational speed ($\Omega>\gamma$).

\subsection{Low rotational speed}\label{sec:low}
In this subsection, we consider the low rotational speed $0<\Omega <\gamma$. Let us start with the following observation, mimicking the
proof of \cite[Proposition~2.1]{ArNeSp19} (see also \cite[Theorem~7.8]{BHHZ-p}):

\begin{lemma} \label{lem:norm}
	Let $\gamma,\gamma_0>0$ and $V_0 \ge 0$. If $0<\Omega<\gamma$, then for any $f \in \Sigma$,
	\begin{align} \label{eq:norm}
	\|\nabla f\|^2_{L^2} + 2\int_{\R^2} V|f|^2dx-2 \Omega L(f) \simeq \|\nabla f\|^2_{L^2} + \|xf\|^2_{L^2}.
	\end{align}
\end{lemma}

\begin{proof}
	We first observe that by Cauchy-Schwarz and Young inequalities, we have for any $\delta>0$,
	\begin{align}
	\Omega |L(f)| &\le \Omega \left(\|x_1 f\|_{L^2} \|\partial_{x_2} f\|_{L^2} + \|x_2 f\|_{L^2} \|\partial_{x_1} f\|_{L^2} \right) \nonumber \\
	&\le \Omega \|xf\|_{L^2} \|\nabla f\|_{L^2}\nonumber \\
	&\le \delta \|\nabla f\|^2_{L^2} + \frac{\Omega^2}{4\delta} \|xf\|^2_{L^2}. \label{eq:Lz}
	\end{align}
	Using \eqref{eq:Lz} with $\delta=\frac{1}{2}$, we have
	\begin{align}
	B(f) &:= \|\nabla f\|^2_{L^2} + 2\int_{\R^2} V|f|^2dx-2 \Omega L(f)  \label{eq:Bf}\\
	&\le 2\|\nabla f\|^2_{L^2} + (\gamma^2+\Omega^2)
   \|xf\|^2_{L^2} + V_0 \|f\|^2_{L^2} \nonumber \\
	&\le \left(2+\frac{V_0}{2}\right)\|\nabla f\|^2_{L^2} + \left(\gamma^2 + \Omega^2+\frac{V_0}{2}\right) \|xf\|^2_{L^2}  \nonumber \\
	&\le C_1 \left(\|\nabla f\|^2_{L^2}+ \|xf\|^2_{L^2}\right), \nonumber
	\end{align}
	where we have used \eqref{eq:uncer} to get the third line. On the other hand, by \eqref{eq:Lz}, we have for any $\delta>0$,
	\begin{align*}
	B(f) \ge (1-2\delta) \|\nabla f\|^2_{L^2} + \left(\gamma^2-\frac{\Omega^2}{2\delta}\right) \|xf\|^2_{L^2} + 2V_0 \int_{\R^2} e^{-\gamma_0 |x|^2} |f(x)|^2 dx.
	\end{align*}
	We choose $\delta>0$ so that
	\[
	\gamma^2- \frac{\Omega^2}{2\delta} =\frac{\gamma^2-\Omega^2}{2} \quad \text{or} \quad \delta = \frac{\Omega^2}{\gamma^2+\Omega^2}.
	\]
	It follows that
	\begin{align*}
	B(f) \ge \frac{\gamma^2-\Omega^2}{\gamma^2+\Omega^2} \|\nabla f\|^2_{L^2} + \frac{\gamma^2-\Omega^2}{2} \|xf\|^2_{L^2} \ge C_2 \left(\|\nabla f\|^2_{L^2} + \|xf\|^2_{L^2} \right).
	\end{align*}
	The proof is complete.
\end{proof}

\noindent {\it Proof of Theorem \ref{theo:orbi-low}.}
	The proof is divided into two steps.
	
	{\bf Step 1.} We show the existence of minimizers for $I_\Omega(\rho)$. Let $\rho>0$ and $f \in \Sigma$ satisfy $\|f\|^2_{L^2} =\rho$.
	
	We first show that $I_\Omega(\rho)$ is well-defined.  We have
	\begin{align}
	E_\Omega(f) & = \frac{1}{2} B(f) + \frac{1}{2} \int_{|f|^2 >\sqrt{e}} |f|^4 \ln\left(\frac{|f|^2}{\sqrt{e}} \right) dx - \frac{1}{2}\int_{|f|^2 <\sqrt{e}} |f|^4 \ln \left(\frac{\sqrt{e}}{|f|^2}\right) dx \nonumber \\
	&\ge \frac{1}{2} B(f) - \frac{\sqrt{e}}{2} \int_{|f|^2<\sqrt{e}} |f|^2 dx \nonumber \\
	&\ge \frac{1}{2} B(f) - \frac{\sqrt{e}}{2} \rho, \label{eq:est-E}
	\end{align}
	where $B(f)$ is as in \eqref{eq:Bf}. Here we have used the fact that $\ln(1/\lambda) < 1/\lambda$ for $0<\lambda<1$. Using \eqref{eq:norm}, we see that $I_\Omega(\rho) >-\infty$.
	
	Next let $(f_n)_{n\ge 1}$ be a minimizing sequence for $I_\Omega(\rho)$. By \eqref{eq:est-E} and \eqref{eq:norm}, we see that $(f_n)_{n\ge 1}$ is a bounded sequence in $\Sigma$. Thus there exists $\phi \in \Sigma$ such that (up to a subsequence) $f_n \to\phi$ weakly in $\Sigma$ and strongly in $L^r(\R^2)$ for all $2\le r <\infty$ (see e.g., \cite[Theorem XIII.67]{ReedSimon4} for the compactness of the embedding $\Sigma \hookrightarrow L^r(\R^2)$). It follows that
	\begin{align} \label{eq:norm-phi}
	\|\phi\|^2_{L^2} = \lim_{n\rightarrow \infty} \|f_n\|^2_{L^2} =\rho.
	\end{align}
	By the strong convergence, H\"older's inequality, and the fact that
	\begin{align} \label{est-fg}
	\left| |f|^4 \ln \left(\frac{|f|^2}{\sqrt{e}}\right) - |g|^4 \ln \left(\frac{|g|^2}{\sqrt{e}}\right) \right| \le C |f-g| \left( |f|^2 + |f|^4 + |g|^2 +|g|^4\right),
	\end{align}
	we infer that
	\[
	\int_{\R^2} |\phi|^4 \ln \left( \frac{|\phi|^2}{\sqrt{e}}\right) dx = \lim_{n\rightarrow \infty} \int_{\R^2} |f_n|^4 \ln \left(\frac{|f_n|^2}{\sqrt{e}} \right) dx.
	\]
	Moreover, by the lower semicontinuity of the weak convergence and \eqref{eq:norm}, we have
	\[
	B(\phi) \le \liminf_{n\rightarrow \infty} B(f_n).
	\]
	It follows that
	\[
	E_\Omega(\phi) \le \liminf_{n\rightarrow \infty} E_\Omega(f_n) = I_\Omega(\rho)
	\]
	which together with \eqref{eq:norm-phi} show that $\phi$ is a minimizer for $I_\Omega(\rho)$. In view of Lemma~\ref{lem:norm}, we also have that (up to a subsequence)  $f_n$ converges to $\phi$ strongly in $\Sigma$.
	
	{\bf Step 2.} We show the orbital stability. We argue by contradiction. Suppose that $\G_\Omega(\rho)$ is not orbitally stable. There exist $\epsilon_0>0$ and $\phi_0 \in \G_\Omega(\rho)$ and a sequence $u_{0,n} \in \Sigma$ satisfying
	\begin{align} \label{eq:orbi-1}
	\lim_{n\rightarrow \infty} \|u_{0,n} - \phi_0\|_{\Sigma} =0
	\end{align}
	and a sequence of time $(t_n)_{n\ge 1} \subset \R$ such that
	\begin{align} \label{eq:orbi-2}
	\inf_{\phi \in \G_\Omega(\rho)} \|u_n(t_n) - \phi\|_{\Sigma} \ge \epsilon_0,
	\end{align}
	where $u_n$ is the solution to \eqref{eq:nls} with initial data $u_n(0) = u_{0,n}$.
	
	Since $\phi_0 \in \G_\Omega(\rho)$, we have $E_\Omega(\phi_0) = I_\Omega(\rho)$. By \eqref{eq:orbi-1}, \eqref{eq:norm} and Sobolev embedding, we have
	\[
	\|u_{0,n}\|^2_{L^2} \Tend n \infty\|\phi_0\|^2_{L^2} = \rho,
        \quad E_\Omega(u_{0,n}) \Tend n \infty E_\Omega(\phi_0) = I_\Omega(\rho).
	\]
	By the conservation of mass and energy, we get
	\[
	\|u_n(t_n)\|^2_{L^2} \Tend n \infty \rho, \quad
        E_\Omega(u_n(t_n)) \Tend n \infty I_\Omega(\rho) .
	\]
This shows that $(u_n(t_n))_{n\ge 1}$ is a minimizing sequence for $I_\Omega(\rho)$. By Step~1, we see that up to a subsequence, $u_n(t_n) \rightarrow \phi$ strongly in $\Sigma$ for some $\phi \in \G_\Omega(\rho)$ which contradicts \eqref{eq:orbi-2}. The proof is now complete.
\hfill $\Box$

\subsection{Critical rotational speed}\label{sec:critical}
In this subsection, we study the existence and stability of standing waves related to \eqref{eq:ground} with a critical rotational speed $\Omega =\gamma$  (lowest Landau level). To this end, we first collect some basic properties of the magnetic Sobolev space $H^1_A(\R^2)$ (see \eqref{eq:H1A}) in the following lemma.

\begin{lemma} [\cite{EsLi89, LiLo01}] \label{lem:H1A}
	We have:
	\begin{itemize}
		\item $H^1_A(\R^2)$ is a Hilbert space.
		\item $C^\infty_0(\R^2)$ is dense in $H^1_A(\R^2)$.
		\item $H^1_A(\R^2)$ is continuously embedded in $L^r(\R^2)$ for all $2\le r<\infty$.
		\item $H^1_A(\R^2) \subset H^1_{\loc}(\R^2)$.
		\item Diamagnetic inequality:
		\begin{align} \label{eq:mag-ine}
		|\nabla |f|(x)| \le |\nabla_A f(x)| \quad \text{a.e. } x \in \R^2.
		\end{align}
		\item Magnetic Gagliardo-Nirenberg inequality. For
                  $2\le r<\infty$,
		\begin{align} \label{eq:mag-GN}
		\|f\|^{r}_{L^r} \le C_r \|\nabla_A f\|^{r-2}_{L^2} \|f\|^2_{L^2}, \quad \forall f \in H^1_A(\R^2).
		\end{align}
	\end{itemize}
      \end{lemma}
\begin{lemma}\label{lem:GNA}
The best constant in \eqref{eq:mag-GN} is the same as the optimal constant in Gagliardo-Nirenberg inequality for the case $A=0$. In particular, when $r=4$, the best constant in \eqref{eq:mag-GN} can be taken as
      \[
      C_4 = \|Q\|_{L^2}^{-2},
      \]
where $Q$ is the cubic (positive, radial) ground state given by \eqref{eq:Q}. Moreover, the equality in \eqref{eq:mag-GN} cannot be attained.
\end{lemma}

\begin{proof}
We recall the standard Gagliardo-Nirenberg inequality
\begin{equation} \label{eq:GN}
\|f\|_{L^r}^r = \| \lvert f\rvert\|_{L^r}^r\le C_r^{\rm GN} \|f\|_{L^2}^2 \|\nabla |f|\|_{L^2}^{r-2},
\end{equation}
where $ C_r^{\rm GN}$ stands for the best constant in the ``regular'' Gagliardo-Nirenberg inequality. This best constant was first
characterized in \cite{We83}, in terms of  a suitable ground state solution of some elliptic equation. In particular, $C_4^{\rm GN}$ is
given by
\[
C_4^{\rm GN}=\|Q\|_{L^2}^{-2},
\]
where $Q$ is the (only) positive, radially symmetric solution to \eqref{eq:Q} (recall that unlike in \cite{We83},  there is a $\tfrac{1}{2}$ factor in front of the Laplacian in \eqref{eq:Q}, hence  a slightly different formula). As a consequence of the diamagnetic inequality \eqref{eq:mag-ine},
we infer
\begin{equation*}
\|f\|_{L^r}^r\le C_r^{\rm GN}
\|f\|_{L^2}^2 \|\nabla_A f\|_{L^2}^{r-2}.
\end{equation*}
This shows that \eqref{eq:mag-GN} holds and $C_r \le C_r^{\rm GN}$. To see $C_r^{\rm GN} \le C_r$, we follow the argument of \cite[Proposition 3.1]{BHIM}. Note that the argument given \cite{BHIM} does not hold in two dimensions. Denote $A^\sigma(x):= \sigma A(\sigma x)$ with $\sigma>0$. By \eqref{eq:mag-GN}, we have for any smooth compactly supported function $f$,
\[
\|f\|^r_{L^r} \le C_r \|\nabla_{A^\sigma} f\|^{r-2}_{L^2} \|f\|^2_{L^2}, \quad \forall \sigma>0.
\]
On the other hand, by the triangle inequality, we have
\[
\|\nabla_{A^\sigma} f\|_{L^2} \le \|\nabla f\|_{L^2} + \|A^\sigma f\|_{L^2}.
\]
We estimate
\begin{align*}
\int_{\R^2} |A^\sigma (x)f(x)|^2 dx &\le \left(\int_{|x|\le R} |A^\sigma(x)|^2 dx\right)  \|f\|_{L^\infty} \\
&= \left(\int_{|x| \le R\sigma} |A(x)|^2 dx\right) \|f\|_{L^\infty} \rightarrow 0 \quad \text{ as } \sigma \rightarrow 0,
\end{align*}
where $R>0$ is such that $\supp(f)\subset \{ x\in \R^2 \ : \ |x| \le R\}$. This shows that
\[
\|f\|^r_{L^r} \le C_r \|\nabla f\|^{r-2}_{L^2} \|f\|^2_{L^2},
\]
hence $C_r^{\rm GN} \le C_r$.

Finally we show that the equality in \eqref{eq:mag-GN} cannot be achieved. Assume by contradiction that \eqref{eq:mag-GN} is attained by some function $\phi$. By \eqref{eq:mag-ine} and \eqref{eq:GN}, we have
\begin{align*}
\|\phi\|^r_{L^r} = C_r \|\nabla_A \phi\|^{r-2}_{L^2} \|\phi\|^2_{L^2} \ge C_r^{\rm GN} \|\nabla |\phi|\|^{r-2}_{L^2} \|\phi\|^2_{L^2}  \ge \|\phi\|^r_{L^r}
\end{align*}
which implies that
\[
\|\nabla_A \phi\|_{L^2} = \|\nabla|\phi|\|_{L^2}.
\]
This together with the following estimate of \cite[Theorem 7.21]{LiLo01}:
\[
|\nabla |\phi|| = \left| \RE \left( \nabla \phi \frac{\overline{\phi}}{|\phi|} \right)\right| = \left|\RE \left( (\nabla - i A) \phi \frac{\overline{\phi}}{|\phi|} \right) \right| \le |(\nabla -iA)\phi|
\]
imply that
\[
\IM \left((\nabla - i A) \phi \frac{\overline{\phi}}{|\phi|} \right) =0 \Longleftrightarrow A= \IM\left( \frac{\nabla \phi}{\phi}\right).
\]
In particular, $B_{12}(x):=\partial_1 A_2(x) - \partial_2 A_1(x) =0$, where $A(x) = (A_1(x), A_2(x))=\gamma(-x_2,x_1)$. This is a contradiction since $B_{12}(x)=2$. The proof is complete.
\end{proof}

Note that in view of \cite[Lemma 1]{GS} (recall again the $\tfrac{1}{2}$ factor \eqref{eq:Q}),
	\begin{equation}\label{eq:estQ}
	\pi \le \|Q\|^2_{L^2} \le  \pi \ln 2.
	\end{equation}
We also have the following useful remark.
\begin{remark} \label{rem:H1A}
	Let $(f_n)_{n\ge 1}$ be a bounded sequence in $H^1_A(\R^2)$. Then up to a subsequence, $f_n \rightharpoonup \phi$ weakly in $H^1_A(\R^2)$, $f_n \rightarrow \phi$ a.e. in $\R^2$ and $f_n \rightarrow \phi$ strongly in $L^r_{\loc}(\R^2)$ for all $2\le r<\infty$. Moreover,
	\[
	\|\nabla_A f_n\|^2_{L^2} = \|\nabla_A \phi\|^2_{L^2} + \|\nabla_A(f_n-\phi)\|^2_{L^2} + o_n(1).
	\]
\end{remark}

We have the following conditional result on the existence and stability of prescribed mass standing waves to \eqref{eq:nls}.
\begin{proposition} \label{prop:orbi-cri-1}
	Let $K_3=V_0 = 0, \gamma>0$ and $\Omega=\gamma$. Let $0<\rho\le \|Q\|_{L^2}^2$ and assume that $I^0_\gamma(\rho)<0$. Then there exists $\phi \in H^1_A(\R^2)$ such that $E^0_\gamma(\phi) = I^0_\gamma(\rho)$ and $\|\phi\|^2_{L^2} =\rho$. 
	Moreover, the set
	\[
	\G^0_\gamma(\rho) := \left\{ \phi \in H^1_A(\R^2) \ : \ E^0_\gamma(\phi)= I^0_\gamma(\rho), \|\phi\|^2_{L^2} =\rho \right\}
	\]
	is orbitally stable under the flow of \eqref{eq:nls} in the sense that for any $\epsilon>0$, there exists $\delta>0$ such that for any initial data $u_0 \in H^1_A(\R^2)$ satisfying
	\[
	\inf_{\phi \in \G^0_\gamma(\rho)} \|u_0 - \phi\|_{H^1_A} <\delta,
	\]
	the corresponding solution to \eqref{eq:nls} exists globally in time and satisfies
	\[
	\inf_{\phi \in \G^0_\gamma(\rho)} \|u(t)-\phi\|_{H^1_A} <\epsilon
	\]
	for all $t\in \R$.
\end{proposition}

\begin{remark}
	In the case $A\equiv 0$ and $\gamma=0$, it was shown in \cite{CaSp-p} that $I(\rho) := I^0_0(\rho)<0$ for any $\rho>0$. This is done by using the scaling $f_\lambda(x) := \lambda f(\lambda x)$ and taking $\lambda>0$ sufficiently small. In our case, showing $I^0_\gamma(\rho)<0$ is more complicated. The above scaling argument does not work due to the presence of magnetic potential $A$.
\end{remark}

\begin{remark}
	The assumption $\rho \le \|Q\|_{L^2}^2$ is probably only technical, due to
        our argument (see after \eqref{eq:cond-rho}).
\end{remark}

Before giving the proof of Proposition \ref{prop:orbi-cri-1}, let us recall the following version of concentration-compactness lemma.

\begin{lemma}[Concentration-compactness lemma \cite{Li84-1,Li84-2}] \label{lem:conc-comp}
	Let $(f_n)_{n\ge 1}$ be a bounded sequence in $H^1_A(\R^2)$ satisfying
	\[
	\|f_n\|^2_{L^2} =\rho
	\]
	for all $n\ge 1$ with some fixed constant $\rho>0$. Then there
        exists a subsequence -- still denoted by $(f_n)_{n\ge 1}$ -- satisfying one of the following three possibilities:
	\begin{itemize}
		\item {\bf Vanishing.}
		$\displaystyle
		\lim_{n\rightarrow \infty} \sup_{y\in \R^2}
                \int_{B(y,R)} |f_n(x)|^2 dx =0\quad\text{for all }R>0.$
		\item {\bf Dichotomy.} There exist $a \in (0,\rho)$ and sequences $(g_n)_{n\ge 1}$, $(h_n)_{n\ge 1}$ bounded in $H^1_A(\R^2)$ such that
		\begin{align} \label{eq:dicho}
		\left\{
		\renewcommand{\arraystretch}{1.2}
		\begin{array}{l}
		\|f_n - g_n-h_n\|_{L^r} \rightarrow 0 \text{ as } n\rightarrow \infty \text{ for all } 2\le r<\infty, \\
		\|g_n\|^2_{L^2} \rightarrow a, \quad \|h_n\|^2_{L^2}\rightarrow \rho-a \text{ as } k \rightarrow \infty,\\
		\dist(\supp(g_n), \supp(h_n)) \rightarrow \infty \text{ as } n\rightarrow \infty, \\
		\liminf_{n\rightarrow \infty} \|\nabla_A f_n\|^2_{L^2} - \|\nabla_A g_n\|^2_{L^2} - \|\nabla_A h_n\|^2_{L^2} \ge 0.
		\end{array}
		\right.
		\end{align}
		\item {\bf Compactness.} There exists a sequence $(y_n)_{n\ge 1} \subset \R^2$ such that for all $\epsilon>0$, there exists $R(\epsilon)>0$ such that for all $k\ge 1$,
		\[
		\int_{B(y_n, R(\epsilon))} |f_n(x)|^2 dx \ge \rho-\epsilon.
		\]
	\end{itemize}
\end{lemma}
The proof of this result is very similar to that of \cite{Li84-1,
  Li84-2}, see also \cite[Proposition~1.7.6]{CazCourant}. The only difference is the last inequality in \eqref{eq:dicho} which is proved by using the fact that
\[
\|\nabla_A(\varphi_R f_n)\|^2_{L^2} - \|\varphi_R \nabla_A f_n\|^2_{L^2} \le C M R^{-1},
\]
where $\varphi_R(x) = \varphi(x/R)$ with a suitable function $\varphi \in C^\infty(\R^2)$ and $M:= \sup_{n\ge 1} \|f_n\|^2_{H^1_A}$.

\begin{remark}
	\begin{itemize}
		\item If  vanishing occurs, then we infer that $f_n \rightarrow 0$ strongly in $L^r(\R^2)$ for all $2<r<\infty$. Indeed, by the magnetic inequality \eqref{eq:mag-ine}, we see that $|f_n|$ is a bounded sequence in $H^1(\R^2)$. The result follows by applying \cite[Lemma 1.1]{Li84-2} to $(|f_n|)_{n\ge 1}$.
		\item If  compactness occurs, then $(f_n)_{n\ge 1}$ is relatively compact in $L^r(\R^2)$ for all $2\le r<\infty$ up to a translation and change of gauge, i.e. there exists $(y_n)_{n\ge 1} \subset \R^2$ such that up to subsequence
		\[
		e^{i \bar{A}_n} f_n(\cdot + y_n) \rightarrow \phi
		\]
		strongly in $L^r(\R^2)$ for all $2\le r<\infty$, where $\bar{A}_n(x)= A(y_n) \cdot x$ and $\phi \in H^1_A(\R^2)$. In fact,
		denote
		\[
		\tilde{f}_n(x):= e^{iA(y_n) \cdot x} f_n(x+y_n).
		\]
		We see that
		\[
		|\tilde{f}_n(x)|= |f_n(x+y_n)|, \quad \|\nabla_A \tilde{f}_n\|^2_{L^2} = \|\nabla_A f_n\|^2_{L^2}.
		\]
		It follows that $(\tilde{f}_n)_{n\ge 1}$ is a bounded sequence in $H^1_A(\R^2)$ satisfying for all $\epsilon>0$, there exists $R(\epsilon)>0$ such that for all $n\ge 1$,
		\[
		\int_{B^c(0,R(\epsilon))} |\tilde{f}_n(x)|^2dx = \int_{B^c(0,R(\epsilon))} |f_n(x+y_n)|^2 dx<\epsilon.
		\]
		By the standard diagonalization argument and the compact embedding $H^1_A(\R^2) \hookrightarrow L^r(B(0,R))$ for all $R>0$ and all $2\le r<\infty$, we show that $\tilde{f}_n \rightarrow \phi$ strongly in $L^2(\R^2)$, hence strongly in $L^r(\R^2)$ for all $2\le r<\infty$ by Sobolev embedding.
	\end{itemize}
\end{remark}

\noindent {\it Proof of Proposition \ref{prop:orbi-cri-1}.} The proof is divided into three steps.

{\bf Step 1.} We first show that for each $\rho>0$, $I^0_\gamma(\rho)$
is well-defined, i.e. $I^0_\gamma(\rho)>-\infty$. Indeed, recalling
\eqref{eq:E0} and arguing as
in \eqref{eq:est-E}, we have directly
\begin{align} \label{eq:est-E-Omega}
E^0_\gamma(f) \ge \frac{1}{2} \|\nabla_A f\|^2_{L^2} - \frac{\sqrt{e}}{2}\rho
\end{align}
for any $f \in H^1_A(\R^2)$ satisfying $\|f\|^2_{L^2}=\rho$.

{\bf Step 2.} We show the existence of minimizers for
$I^0_\gamma(\rho)$, under the assumption $I^0_\gamma(\rho)<0$. It is done by using the concentration-compactness argument as in \cite{EsLi89}. Let $(f_n)_{n\ge1}$ be a minimizing sequence for $I^0_\gamma(\rho)$. By \eqref{eq:est-E-Omega}, we see that $(f_n)_{n\ge 1}$ is a bounded sequence in $H^1_A(\R^2)$ satisfying
\[
\|f_n\|^2_{L^2}=\rho, \quad \forall n\ge 1, \quad E^0_\gamma(f_n)\rightarrow I^0_\gamma(\rho) \text{ as } n\rightarrow \infty.
\]
By Lemma \ref{lem:conc-comp}, there exists a subsequence still denoted by $(f_n)_{n\ge 1}$ satisfying one of the following three possibilities: vanishing, dichotomy and compactness.

{\bf No vanishing.} Suppose that vanishing occurs. Passing to a subsequence, we infer that $f_n \rightarrow 0$ strongly in $L^r(\R^2)$ for all $2<r<\infty$. Thanks to \eqref{eq:est-log}, we infer that
\[
\int_{\R^2} |f_n|^4 \ln \left(\frac{|f_n|^2}{\sqrt{e}}\right) dx \Tend
n \infty 0,
\]
hence
\[
I^0_\gamma(\rho) = \lim_{n\rightarrow \infty} E^0_\gamma(f_n) = \lim_{n\rightarrow \infty} \|\nabla_A f_n\|^2_{L^2} \ge 0,
\]
which contradicts the assumption $I^0_\gamma(\rho)<0$.

{\bf No dichotomy.} If dichotomy occurs, then there exist $a \in (0,\rho)$ and sequences $(g_n)_{n\ge 1}$, $(h_n)_{n\ge 1}$ bounded in $H^1_A(\R^2)$ such that \eqref{eq:dicho} holds. To rule out the dichotomy, we first claim that there exists $\delta>0$ such that
\begin{align} \label{eq:no-dicho-1}
\liminf_{n\rightarrow \infty} \|f_n\|^4_{L^4} \ge \delta>0.
\end{align}
Since $E^0_\gamma(f_n) \rightarrow I^0_\gamma(\rho)<0$, we see that for $n$ sufficiently large, $E^0_\gamma(f_n) \le \frac{I^0_\gamma(\rho)}{2}$. It follows that
\begin{align*}
\frac{I^0_\gamma(\rho)}{2} \ge E^0_\gamma(f_n) = \frac{1}{2} \|\nabla_A f_n\|^2_{L^2}  &+ \frac{1}{2} \int_{|f_n|^2>\sqrt{e}} |f_n|^4 \ln \left(\frac{|f_n|^2}{\sqrt{e}}\right) dx \\
&- \frac{1}{2} \int_{|f_n|^2<\sqrt{e}} |f_n|^4 \ln \left(\frac{\sqrt{e}}{|f_n|^2}\right) dx
\end{align*}
which implies, since $0<\ln z\lesssim \sqrt z$ for $z>1$,
\[
\|f_n\|^3_{L^3} \gtrsim \int_{|f_n|^2 <\sqrt{e}} |f_n|^4 \ln \left(\frac{\sqrt{e}}{|f_n|^2}\right) dx \ge -\frac{I^0_\gamma(\rho)}{2}>0,
\]
for $n$ sufficiently large. The claim \eqref{eq:no-dicho-1} follows from the H\"older inequality $\|f\|^3_{L^3} \le \|f\|^2_{L^4} \|f\|_{L^2}$ and $\|f_n\|^2_{L^2} =\rho$.

We next claim that
\begin{equation} \label{eq:no-dicho-2}
\liminf_{n\rightarrow \infty}\( E^0_\gamma(f_n) - E^0_\gamma(g_n) - E^0_\gamma(h_n) \)\ge 0.
\end{equation}
To see this, we consider $K(z) = z^4 \ln (z^2)$ for $z>0$. Using  Taylor expansion and the fact that
\[
|K'(z)| \lesssim_\epsilon |z|^{3-\epsilon} + |z|^{3+\epsilon}
\]
for any $\epsilon>0$, we infer that
\[
\left|K\left(\sum_{j=1}^N z_j \right) - \sum_{j=1}^N K(z_j)\right| \lesssim_{\epsilon,N} \sum_{\ell \ne k} |z_\ell| (|z_k|^{3-\epsilon} + |z_k|^{3+\epsilon}).
\]
Applying the above estimate with $\epsilon=1, N=3$ and $e_n:= f_n-g_n - h_n $, we see that
\begin{multline*}
\left| \int K(f_n) dx - \int K(g_n)dx -\int K(h_n) dx -\int K(e_n) dx\right| \\ \lesssim \int |e_n| (|g_n|^2+ |g_n|^4+ |h_n|^2+|h_n|^4)dx + \int (|g_n| + |h_n|)(|e_n|^2 + |e_n|^4) dx.
\end{multline*}
Using \eqref{eq:est-log}, H\"older inequality and the fact that $e_n \rightarrow 0$ strongly in $L^r(\R^2)$ for all $2\le r<\infty$, we have
\[
\int_{\R^2} |f_n|^4 \ln(|f_n|^2) dx - \int_{\R^2} |g_n|^4 \ln(|h_n|^2)
dx - \int_{\R^2} |h_n|^4 \ln(|h_n|^2) dx \Tend n \infty 0.
\]
Similarly, we can prove that
\[
\|f_n\|^4_{L^4} - \|g_n\|^4_{L^4} - \|h_n\|^4_{L^4} \Tend n \infty 0.
\]
This shows that
\[
\int_{\R^2} |f_n|^4 \ln \left(\frac{|f_n|^2}{\sqrt{e}} \right) dx -
\int_{\R^2} |g_n|^4 \ln \left(\frac{|g_n|^2}{\sqrt{e}} \right)
dx-\int_{\R^2} |h_n|^4 \ln \left(\frac{|h_n|^2}{\sqrt{e}} \right) dx
\Tend n\infty 0.
\]
Using this and \eqref{eq:dicho}, we show \eqref{eq:no-dicho-2}.

Now, let $\lambda>0$. We have
\begin{align*}
E^0_\gamma(\lambda f) &= \frac{\lambda^2}{2} \|\nabla_A f\|^2_{L^2} + \frac{\lambda^4}{2} \int_{\R^2} |f|^4 \ln \left(\lambda^2 \frac{|f|^2}{\sqrt{e}}\right) dx \\
&= \frac{\lambda^2}{2} \|\nabla_A f\|^2_{L^2} + \frac{\lambda^4}{2} \int_{\R^2} |f|^4 \ln \left(\frac{|f|^2}{\sqrt{e}}\right) dx  + \frac{\lambda^4 \ln(\lambda^2)}{2} \|f\|^4_{L^4} \\
&= \lambda^4 E^0_\gamma(f) - \frac{\lambda^4-\lambda^2}{2} \|\nabla_A f\|^2_{L^2} + \frac{\lambda^4 \ln(\lambda^2)}{2} \|f\|^4_{L^4},
\end{align*}
which implies
\[
E^0_\gamma(f) = \lambda^{-4} E^0_\gamma(\lambda f) + \frac{1-\lambda^{-2}}{2} \|\nabla_A f\|^2_{L^2} - \frac{\ln(\lambda^2)}{2}\|f\|^4_{L^4}.
\]
Set $\lambda_n:=\frac{\sqrt{\rho}}{\|g_n\|_{L^2}}$ and $\mu_n := \frac{\sqrt{\rho}}{\|h_n\|_{L^2}}$. It follows that $\|\lambda_n g_n\|^2_{L^2} =\|\mu_n h_n\|^2_{L^2} =\rho$, hence $E^0_\gamma(\lambda_n g_n), E^0_\gamma(\mu_n h_n) \ge I^0_\gamma(\rho)$. We see that for $n$ sufficiently large,
\begin{align*}
E^0_\gamma(g_n) &=  \lambda_n^{-4} E^0_\gamma(\lambda_n g_n) + \frac{1-\lambda_n^{-2}}{2} \|\nabla_A g_n\|^2_{L^2} - \frac{\ln(\lambda_n^2)}{2} \|g_n\|^4_{L^4} \\
& \ge \lambda_n^{-4} I^0_\gamma(\rho) + \frac{1-\lambda_n^{-2}}{2C_4} \frac{\|g_n\|^4_{L^4}}{\|g_n\|^2_{L^2}} - \frac{\ln(\lambda_n^2)}{2} \|g_n\|^4_{L^4}.
\end{align*}
Here we have used the magnetic Gagliardo-Nirenberg inequality \eqref{eq:mag-GN} and the fact $\lambda_n \rightarrow \sqrt{\frac{\rho}{a}}>1$ as $n\rightarrow \infty$. A similar estimate goes for $E^0_\gamma(h_n)$. Thus, we get
\begin{align*}
E^0_\gamma(g_n) + E^0_\gamma(h_n) \ge \left(\lambda_n^{-4}+\mu_n^{-4}\right) I^0_\gamma(\rho) &+ \frac{1-\lambda_n^{-2}}{2C_4} \frac{\|g_n\|^4_{L^4}}{\|g_n\|^2_{L^2}} - \frac{\ln(\lambda_n^2)}{2} \|g_n\|^4_{L^4} \\
&+ \frac{1-\mu_n^{-2}}{2C_4} \frac{\|h_n\|^4_{L^4}}{\|h_n\|^2_{L^2}}- \frac{\ln(\mu_n^2)}{2} \|h_n\|^4_{L^4}.
\end{align*}
Passing $n\rightarrow \infty$ and using \eqref{eq:no-dicho-2}, we obtain
\begin{align}
I^0_\gamma(\rho) &\ge \frac{a^2+ (\rho-a)^2}{\rho^2} I^0_\gamma(\rho) + \frac{1}{2} \left(\frac{1}{C_4\rho}\left(\frac{\rho}{a} -1\right) - \ln\left(\frac{\rho}{a}\right)\right) \liminf_{n\rightarrow \infty} \|g_n\|^4_{L^4} \nonumber \\
&\mathrel{\phantom{\ge \frac{a^2+ (\rho-a)^2}{\rho^2} I^0_\gamma(\rho) }} + \frac{1}{2} \left(\frac{1}{C_4\rho}\left(\frac{\rho}{\rho-a} -1\right) - \ln\left(\frac{\rho}{\rho-a}\right)\right) \liminf_{n\rightarrow \infty} \|h_n\|^4_{L^4} \nonumber \\
&\ge \frac{a^2+ (\rho-a)^2}{\rho^2} I^0_\gamma(\rho)  + \frac{1}{2} \min \left\{K_1, K_2\right\} \liminf_{n\rightarrow \infty} \|f_n\|^4_{L^4}, \label{eq:no-dicho-3}
\end{align}
where
\begin{equation} \label{eq:cond-rho}
K_1:= \frac{1}{C_4\rho}\left(\frac{\rho}{a} -1\right) - \ln\left(\frac{\rho}{a}\right), \quad K_2:= \frac{1}{C_4\rho}\left(\frac{\rho}{\rho-a} -1\right) - \ln\left(\frac{\rho}{\rho-a}\right).
\end{equation}
Here $K_1, K_2 >0$ as soon as $C_4\rho\le 1$, which follows from
$0<\rho \le \|Q\|_{L^2}^2$, see Lemma~\ref{lem:GNA}. Using \eqref{eq:no-dicho-1} and \eqref{eq:no-dicho-3}, we infer that $I^0_\gamma(\rho)>0$ which is a contradiction.

{\bf Compactness.} Therefore,  compactness must occur. In this case, there exist $\phi \in H^1_A(\R^2)$ and $(y_n)_{n\ge 1} \subset \R^2$ such that up to a subsequence,
\[
\tilde{f}_n \rightarrow \phi
\]
strongly in $L^r(\R^2)$ for all $2\le r<\infty$, where $\tilde{f}_n(x):= e^{i A(y_n) \cdot x} f_n(x+y_n)$. It follows that
\[
\|\phi\|^2_{L^2} = \lim_{n\rightarrow \infty} \|\tilde{f}_n\|^2_{L^2} = \lim_{n\rightarrow \infty} \|f_n(\cdot+y_n)\|^2_{L^2} = \rho,
\]
and by \eqref{eq:est-log},
\[
\int_{\R^2} |\phi|^2 \ln\left(\frac{|\phi|^2}{\sqrt{e}}\right) dx = \lim_{n\rightarrow \infty} \int_{\R^2} |\tilde{f}_n|^4 \ln \left( \frac{|\tilde{f}_n|^2}{\sqrt{e}}\right) dx = \lim_{n\rightarrow\infty} \int_{\R^2} |f_n|^4 \ln \left(\frac{|f_n|^2}{\sqrt{e}}\right) dx.
\]
We also have
\[
\|\nabla_A\phi\|^2_{L^2} \le \liminf_{n\rightarrow \infty} \|\nabla_A \tilde{f}_n\|^2_{L^2} =\liminf_{n\rightarrow \infty} \|\nabla_A f_n\|^2_{L^2}.
\]
Thus, we get
\[
E^0_\gamma(\phi) \le \liminf_{n\rightarrow \infty} E^0_\gamma(f_n) = I^0_\gamma(\rho).
\]
This shows that $\phi$ is a minimizer for $I^0_\gamma(\rho)$. We also have that $\tilde{f}_n \rightarrow \phi$ strongly in $H^1_A(\R^2)$ as $n\rightarrow \infty$.

{\bf Step 3.} The orbital stability of $\G^0_\gamma(\rho)$ follows by the same argument as in Theorem \ref{theo:orbi-low}. We thus omit the details. The proof is now complete.
\hfill $\Box$

In the following lemma, we show that the conditions $0<\rho \le \|Q\|_{L^2}^2$ and $I^0_\gamma(\rho)<0$ are satisfied for some data $f\in H^1_A(\R^2)$.
\begin{lemma} \label{lem:data}
	If $0<\gamma<\frac{1}{2 e^{3/2}}$, then there exists $f \in H^1_A(\R^2)$ satisfying $\|f\|^2_{L^2} = \rho \le \|Q\|_{L^2}^2$ and $I^0_\gamma(\rho) <0$.
\end{lemma}

\begin{proof}
We look for a function $f\in H^1_A(\R^2)$ satisfying
\begin{align} \label{eq:data}
\|f\|^2_{L^2}=\rho \le \|Q\|_{L^2}^2, \quad E^0_\gamma(f) <0.
\end{align}
To this end, we consider $f_b(x) = e^{-b|x|^2}$ for some $b>0$ to be chosen later. A direct computation shows
\[
\|f_b\|^2_{L^2} = \int_{\R^2} e^{-2b|x|^2} dx = \frac{\pi}{2b}.
\]
We also have
\[
\|\nabla_A f_b\|^2_{L^2} = (\gamma^2+ 4b^2) \int_{\R^2} |x|^2 e^{-2b|x|^2} dx = \pi \left( 1+\frac{\gamma^2}{4b^2}\right).
\]
In addition, we have
\[
\|f_b\|^4_{L^4(\R^2)} = \int_{\R^2} e^{-4b|x|^2} dx = \frac{\pi}{4b}
\]
and
\[
\int_{\R^2} |f_b|^4 \ln \left(\frac{|f_b|^2}{\sqrt{e}}\right) dx = - 2b \int_{\R^2} |x|^2 e^{-4b|x|^2} dx -\frac{1}{2} \int_{\R^2} e^{-4b|x|^2} dx = -\frac{\pi}{4b}.
\]
Let $\lambda>0$ to be chosen later. We see that
\[
\|\lambda f_b\|^2_{L^2} =\lambda^2 \|f_b\|^2_{L^2} = \frac{\pi}{2b} \lambda^2, \quad E^0_\gamma(\lambda f_b) = \frac{\pi}{2} \left(1+\frac{\gamma^2}{4b^2}\right) \lambda^2 - \frac{\pi}{8b} \lambda^4 + \frac{\pi}{8b} \lambda^4 \ln(\lambda^2).
\]
To make $\|\lambda f_b\|^2_{L^2} \le \|Q\|_{L^2}^2$, we need
$\lambda^2 \le \frac{2b}{\pi}\|Q\|_{L^2}^2$. In view of
\eqref{eq:estQ}, this is granted by the property $\lambda^2\le 2b$. Consider
\[
F(\theta) = \frac{\pi}{2} \left(1+\frac{\gamma^2}{4b^2}\right) \theta - \frac{\pi}{8b} \theta^2 + \frac{\pi}{8b} \theta^2 \ln (\theta), \quad \text{ for } 0<\theta \le 2b.
\]
We rewrite
\[
F(\theta) = \frac{\pi \theta}{8b^2} \left( 4b^2+ \gamma^2 - b\theta + b\theta \ln(\theta)\right) = :\frac{\pi \theta}{8b^2} G(\theta).
\]
The condition $E^0_\gamma(\lambda f_b)<0$ is now reduced to finding $\theta \in \left(0,2b\right]$ so that $G(\theta)<0$. We have
\[
G'(\theta) = b \ln (\theta).
\]
If $2b\ge 1$, we see that $G$ attains its minimum at $\theta =1$,
however $G(1) =4b^2+\gamma^2-b>0$ for $2b\ge 1$. So, we need $2b<1$. In this case, $G$ is strictly decreasing on $\left(0,2b\right]$. To find $\theta \in \left(0,2b\right]$ so that $G(\theta)<0$, a necessary condition is
\[
G\left(2b\right) <0 \Longleftrightarrow 4b^2+ \gamma^2 - 2b^2 -2b^2\ln
\left(\frac{1}{2b}\right) <0.
\]
Consider $H(b):= 2b^2+ \gamma^2 +2b^2\ln (2b)$ for $b\in \left(0,\frac{1}{2}\right)$.  We compute
\[
H'(b) = 2b\(3+2\ln(2b)\).
\]
We see that on $\left(0,\frac{1}{2}\right)$, $H$ attains its minimum at $b_0 = \frac{e^{-3/2}}{2}<\frac{1}{2}$. A direct computation shows
\[
H(b_0) = \gamma^2-\frac{1}{4 e^3}.
\]
Thus, if $\gamma<\tfrac{1}{ 2 e^{3/2}}$, then we have $G\left(2b_0\right) <0$. Therefore, by choosing $\lambda>0$ so that $\lambda^2 \le 2b_0$, we show the existence of a function $f \in H^1_A(\R^2)$ satisfying \eqref{eq:data}. The proof is complete.	
\end{proof}

\noindent {\it Proof of Theorem \ref{theo:orbi-cri-1}.}
It follows directly from Proposition \ref{prop:orbi-cri-1} and Lemma \ref{lem:data}.
\hfill $\Box$

\smallbreak

In the case $V_0>0$, the concentration-compactness argument presented
above does not work due to the lack of spatial translation in the new
potential term
\[
\int_{\R^2} e^{-\gamma_0|x|^2} |f(x)|^2 dx.
\]
More precisely, the sequence $(y_n)_{n\ge 1} \subset \R^2$, which
appeared in the compactness, may tend to infinity, and then
\[
\int_{\R^2} e^{-\gamma_0|x|^2} |f_n(x+y_n)|^2 dx \Tend n \infty 0.
\]
To overcome this difficulty, we restrict our consideration on $H^1_{A,\rad}(\R^2)$ the space of radially symmetric functions of $H^1_A(\R^2)$. Note that this restriction has the drawback that we no longer see the effect of rotation to the equation since $L_z f=0$ for a radial functions $f$.

\noindent {\it Proof of Theorem \ref{theo:orbi-cri-2}.}
	We proceed in two steps.
	
	{\bf Step 1.} We show that there exists a minimizer for $I_{\gamma,\rad}(\rho)$. Let $\rho>0$ and $f \in H^1_{A,\rad}(\R^2)$ satisfy $\|f\|^2_{L^2}=\rho$. We have
	\begin{align}
	E_\gamma(f) &= \frac{1}{2} \|\nabla_A f\|^2_{L^2} + V_0
                      \int_{\R^2} e^{-\gamma_0 |x|^2} |f(x)|^2 dx +
                      \frac{1}{2} \int_{|f|^2>\sqrt{e}} |f|^4 \ln
                      \left(\frac{|f|^2}{\sqrt{e}} \right) dx
                      \nonumber \\
	&\mathrel{\phantom{= \frac{1}{2} \|\nabla_A f\|^2_{L^2} + V_0 \int_{\R^2} e^{-\gamma |x|^2} |f(x)|^2 dx}}- \frac{1}{2}\int_{|f|^2<\sqrt{e}} |f|^4 \ln \left(\frac{\sqrt{e}}{|f|^2}\right) dx \nonumber \\
	&\ge \frac{1}{2} \|\nabla_A f\|^2_{L^2} - \frac{\sqrt{e}}{2} \int_{|f|^2<\sqrt{e}} |f|^2 dx \nonumber \\
	&\ge \frac{1}{2} \|\nabla_A f\|^2_{L^2} -\frac{\sqrt{e}}{2} \rho. \label{est-E-cri}
	\end{align}
	This shows that $I_{\gamma,\rad}(\rho)>-\infty$ is well-defined.
	
	Next, let $(f_n)_{n\ge 1}$ be a minimizing sequence for $I_{\gamma,\rad}(\rho)$. By \eqref{est-E-cri}, $(f_n)_{n\ge 1}$ is a bounded sequence in $H^1_{A,\rad}(\R^2)$. Since $H^1_{A,\rad}(\R^2) \hookrightarrow L^r(\R^2)$ is compact for all $2\le r<\infty$, there exists $\phi \in H^1_{A,\rad}(\R^2)$ such that, up to a subsequence, $f_n \rightarrow \phi$ weakly in $H^1_{A,\rad}(\R^2)$ and strongly in $L^r(\R^2)$ for all $2\le r<\infty$. By the strong convergence and \eqref{est-fg}, we see that
	\[
	\int_{\R^2} |\phi|^4 \ln \left(\frac{|\phi|^2}{\sqrt{e}}\right) dx=\lim_{n\rightarrow \infty} \int_{\R^2} |f_n|^4 \ln \left(\frac{|f_n|^2}{\sqrt{e}}\right) dx.
	\]
	On the other hand, by the weak continuity of the potential energy (see e.g., \cite[Theorem 11.4]{LiLo01}), we have
	\[
	\int_{\R^2} e^{-\gamma_0 |x|^2} |\phi(x)|^2 dx = \lim_{n\rightarrow \infty} \int_{\R^2} e^{-\gamma_0 |x|^2} |f_n(x)|^2 dx.
	\]
	Moreover, the lower semicontinuity of the weak convergence implies
	\[
	\|\nabla_A \phi\|_{L^2} \le \liminf_{n\rightarrow \infty} \|\nabla_A f_n\|_{L^2}.
	\]
	Thus we get
	\[
	E_\gamma(\phi) \le \liminf_{n\rightarrow \infty} E_\gamma(f_n)= I_{\gamma,\rad}(\rho),
	\]
	which together with
	\[
	\|\phi\|^2_{L^2} = \lim_{n\rightarrow \infty} \|f_n\|^2_{L^2} =\rho
	\]
	implies that $\phi$ is a minimizer for $I_{\gamma,\rad}(\rho)$. In addition, we have that (up to a subsequence), $f_n\rightarrow \phi$ strongly in $H^1_{A,\rad}(\R^2)$.
	
	{\bf Step 2.} The orbital stability follows from the same argument as in the proof of Theorem \ref{theo:orbi-low}. We omit the details.
\hfill $\Box$

\subsection{High rotational speed}\label{sec:high}
In this subsection, we show the nonexistence of minimizers for $I_\Omega(\rho)$ given in Theorem \ref{theo:non-exis}.

\noindent {\it Proof of Theorem \ref{theo:non-exis}.}
The proof is inspired by the idea of \cite{BWM} (see also \cite{Ba07}). We consider
\[
f_m(x) =f_m(r,\theta)= C(m,\gamma) r^m e^{-\frac{\gamma|x|^2}{2}} e^{im\theta},
\]
where $m \in \N$ and $C(m,\gamma)>0$ to be determined later. Here $(r,\theta)$ is the polar coordinate of $x=(x_1,x_2) \in \R^2$, i.e.
\[
x_1 = r\cos \theta, \quad x_2 = r\sin \theta, \quad r >0, \quad \theta \in (0,2\pi].
\]
It is useful to recall the following formulas:
\[
  \begin{pmatrix}
    \partial_r x_1 & \partial_\theta x_1 \\
\partial_r x_2 & \partial_\theta x_2
\end{pmatrix}
=
\begin{pmatrix}
  \cos \theta & -r \sin \theta \\
\sin \theta & r \cos \theta
\end{pmatrix},
\]
and
\[
  \begin{pmatrix}
    \partial_{x_1} r & \partial_{x_2} r \\
\partial_{x_1} \theta & \partial_{x_2} \theta
  \end{pmatrix}
  =
  \begin{pmatrix}
    \cos \theta & \sin \theta \\
-\frac{\sin \theta}{r} & \frac{\cos \theta}{r}
  \end{pmatrix}
 .\]
We have
\[
\|f_m\|^2_{L^2(\R^2)} = \int_{\R^2} |f_m(x)|^2 dx = 2\pi C(m,\gamma)^2 \int_0^\infty r^{2m} e^{-\gamma r^2} r dr.
\]
Let $\rho>0$. We choose $C(m,\gamma)$ so that $\|f_m\|^2_{L^2(\R^2)} =\rho$ for all $m \in \N$. Set
\[
I(\gamma, m):= \int_0^\infty r^{2m} e^{-\gamma r^2} r dr,
\]
and write
\[
\|f_m\|^2_{L^2(\R^2)} = 2 \pi C(m,\gamma)^2 I(\gamma, m).
\]
Integrating by parts, we see that $I(\gamma, m) = \frac{m}{\gamma} I(\gamma, m-1)$ which, by induction, implies $I(\gamma,  m) = \frac{m!}{\gamma^m} I(\gamma, 0)$, where
\[
I(\gamma, 0) = \int_0^\infty e^{-\gamma r^2} r dr = \frac{1}{2\gamma}.
\]
Thus we get
\begin{align} \label{Im}
I(\gamma, m) = \frac{m!}{2\gamma^{m+1}}.
\end{align}
The condition $\|f_m\|^2_{L^2(\R^2)} =\rho$ is equivalent to
\begin{align} \label{C-m-gamma}
\rho= 2\pi C(m,\gamma)^2 \frac{m!}{2\gamma^{m+1}} \Longleftrightarrow C(m,\gamma)^2 = \frac{\rho \gamma^{m+1}}{\pi m!}
\end{align}
We compute
\begin{align*}
\partial_{x_1} f_m &= \partial_r f_m  \partial_{x_1} r + \partial_\theta f_m \partial_{x_1} \theta \\
&= C(m,\gamma) e^{im\theta} e^{-\frac{\gamma r^2}{2}} (mr^{m-1} - \gamma r^{m+1}) \cos \theta \\
&\mathrel{\phantom{=}} - C(m,\gamma) e^{im\theta} e^{-\frac{\gamma r^2}{2}} i m r^{m-1} \sin \theta\\
&= C(m,\gamma) e^{im\theta} e^{-\frac{\gamma r^2}{2}} \left((mr^{m-1} - \gamma r^{m+1}) \cos \theta - im r^{m-1} \sin \theta \right),
\end{align*}
and
\begin{align*}
\partial_{x_2} f_m &= \partial_r f_m \partial_{x_2}r + \partial_\theta f_m \partial_{x_2}\theta \\
&= C(m,\gamma) e^{im\theta} e^{-\frac{\gamma r^2}{2}} \left( (m r^{m-1} - \gamma r^{m+1}) \sin \theta + i m r^{m-1} \cos \theta \right).
\end{align*}
It follows that
\begin{align*}
|\nabla f_m|^2 &= |\partial_{x_1} f_m|^2 + |\partial_{x_2} f_m|^2 \\
&= C(m,\gamma)^2 e^{-\gamma r^2} \Big[\left((mr^{m-1} - \gamma r^{m+1}) \cos \theta - im r^{m-1} \sin \theta \right)^2 \\
&\mathrel{\phantom{= [C(m,\gamma)]^2 e^{-\gamma r^2} \Big[}} +
\left( (m r^{m-1} - \gamma r^{m+1}) \sin \theta + i m r^{m-1} \cos \theta \right)^2
\Big] \\
&= C(m,\gamma)^2 e^{-\gamma r^2} \left(2 m^2 r^{2(m-1)} - 2m\gamma r^{2m} + \gamma^2 r^{2(m+1)}\right).
\end{align*}
Thus we have
\begin{align*}
\|\nabla f_m\|^2_{L^2(\R^2)} & = 2\pi \int_0^\infty |\nabla f_m|^2 r dr \\
&= 2\pi C(m,\gamma)^2 \int_0^\infty \left(2m^2 r^{2(m-1)} -2 m\gamma e^{2m} + \gamma^2 r^{2(m+1)} \right) e^{-\gamma r^2} rdr \\
&= 2\pi C(m,\gamma)^2 \left( 2m^2 I(\gamma, m-1) - 2m\gamma I(\gamma, m) + \gamma^2 I(\gamma, m+1)\right).
\end{align*}
Thanks to \eqref{Im} and \eqref{C-m-gamma}, we have
\begin{align*}
\|\nabla f_m\|^2_{L^2(\R^2)} &= 2\pi \frac{\rho \gamma^{m+1}}{\pi m!} \left( 2m^2 \frac{(m-1)!}{2\gamma^m} -2m\gamma \frac{m!}{2\gamma^{m+1}} + \gamma^2 \frac{(m+1)!}{2\gamma^{m+2}}\right) \\
&=\rho(m+1) \gamma.
\end{align*}
We next compute
\begin{align*}
\int_{\R^2} V |f_m|^2 dx &= C(m,\gamma)^2\int_{\R^2} \left( \frac{\gamma^2}{2} |x|^2 + V_0 e^{-\gamma_0 |x|^2} \right) |x|^{2m} e^{-\gamma |x|^2} dx \\
&=2\pi C(m,\gamma)^2 \int_0^\infty \left( \frac{\gamma^2}{2} r^2  +
   V_0 e^{-\gamma_0 r^2}\right) r^{2m} e^{-\gamma r^2} r dr \\
&= 2\pi C(m,\gamma)^2 \left( \frac{\gamma^2}{2} \int_0^\infty r^{2(m+1)} e^{-\gamma r^2} r dr + V_0 \int_0^\infty r^{2m} e^{-(\gamma+\gamma_0) r^2} r dr \right) \\
&= 2\pi C(m,\gamma)^2 \left( \frac{\gamma^2}{2} I(\gamma, m+1) + V_0 I(\gamma+\gamma_0, m)\right),
\end{align*}
which, together with \eqref{Im} and \eqref{C-m-gamma}, implies that
\begin{align*}
\int_{\R^2} V |f_m|^2 dx &= 2\pi \frac{\rho\gamma^{m+1}}{\pi m!} \left( \frac{\gamma^2}{2} \frac{(m+1)!}{2\gamma^{m+2}} + V_0 \frac{m!}{2(\gamma+\gamma_0)^{m+1}}  \right) \\
&= \frac{\rho}{2}(m+1)\gamma + \(\frac{\gamma}{\gamma+\gamma_0}\)^{m+1}V_0 \rho.
\end{align*}
Moreover, since the rotation operator can be rewritten as $L_z = - i \partial_\theta$, we have
\begin{align*}
L(f_m) &= \int_{\R^2} \overline{f}_m L_z f_m dx = 2\pi C(m,\gamma)^2 \int_0^\infty m r^{2m} e^{-\gamma r^2} r dr \\
&=  2\pi m C(m,\gamma)^2 I(\gamma, m) = m \|f_m\|^2_{L^2(\R^2)}= m \rho.
\end{align*}
Next we compute
\begin{align*}
\|f_m\|^6_{L^6(\R^2)} &= 2\pi C(m,\gamma)^6 \int_0^\infty r^{6m} e^{-3\gamma r^2} r dr = 2\pi C(m,\gamma)^6 I\left(3\gamma, 3m\right) \\
&= 2\pi C(m,\gamma)^6  \frac{(3m)!}{2 (3\gamma)^{3m+1}}
= 2\pi \left(\frac{\rho\gamma^{m+1}}{\pi m!}\right)^3  \frac{(3m)!}{2 (3\gamma)^{3m+1}} \\
&= \rho^3 \frac{\gamma^2}{\pi^2} \frac{(3m)!}{(3^m m!)^3},
\end{align*}
which tends to zero as $m \rightarrow \infty$ for each $\rho>0$ fixed. By interpolation between $L^2$ and $L^6$, we infer that
\[
\|f_m\|^p_{L^p(\R^2)} \Tend m \infty 0 ,\quad \forall p\in ]2,6].
\]
This yields
\[
\left| \int_{\R^2} |f_m|^4 \ln \left( \frac{|f_m|^2}{\sqrt{e}}\right)
  dx\right| \lesssim \|f_m\|^3_{L^3(\R^2)} + \|f_m\|^5_{L^5(\R^2)}
\Tend m \infty 0.
\]
Finally we have
\begin{align*}
E_\Omega(f_m) &= \frac{1}{2} \|\nabla f_m\|^2_{L^2} + \int_{\R^2} V |f_m|^2 dx + \frac{1}{2} \int_{\R^2} |f_m|^4 \ln \left(\frac{|f_m|^2}{\sqrt{e}}\right) dx - \Omega L(f_m) \\
&= \rho(m+1)\gamma +\(\frac{\gamma}{\gamma+\gamma_0}\)^{m+1}V_0 \rho + o_m(1) - \Omega m \rho \\
&= - m \rho(\Omega -\gamma) + \rho\gamma+ \(\frac{\gamma}{\gamma+\gamma_0}\)^{m+1}V_0 \rho+ o_m(1),
\end{align*}
where $o_m(1) \rightarrow 0$ as $m \rightarrow \infty$. As $\Omega
>\gamma$, letting $m \rightarrow \infty$ yields $E_\Omega(f_m) \rightarrow -\infty$. Therefore, for any $\rho>0$, there is no minimizer for $I_\Omega(\rho)$.
\hfill $\Box$

\appendix


\section{Ground states}\label{sec:ground}

In this appendix, we collect some information on stationary solutions
in the radial case, where the rotation terms is absent, so
\eqref{eq:ground} becomes
\begin{align} \label{eq:ground-rad}
-\frac{1}{2}\Delta \phi + \omega \phi + \phi |\phi|^2 \ln (|\phi|^2) + V \phi =0, \quad x \in \R^2.
\end{align}
Recall that
\[
V(x) = \frac{\gamma^2}{2}|x|^2 + V_0 e^{-\gamma_0 |x|^2}. 
\]
We introduce the minimizing problem
\begin{align} \label{eq:omega0}
\omega_{V_0} := \inf \left\{ \frac{1}{2} \|\nabla f\|^2_{L^2} + \int_{\R^2} V|f|^2 dx \ : \ f \in \Sigma, \|f\|^2_{L^2} =1\right\}.
\end{align}
Denoting by
\[
H = -\frac{1}{2}\Delta+V,
\]
we see that $\omega_{V_0}$ is simply the bottom of the spectrum of $H$. If $V_0=0$, then it is well-known that $\omega_0$ is attained by the Gaussian
\[
\varphi_0(x)=\sqrt{\frac{\gamma}{\pi}}e^{-\gamma |x|^2/2},\quad\text{and}\quad \omega_0 = \gamma.
\]
Moreover, every minimizer for $\omega_0$ is of the form $f(x) =e^{i\theta} \varphi_0(x)$, where $\theta \in \R$.
\begin{lemma} \label{lem:omega-V0}
	Let $V_0 \ge 0$. Then $\omega_{V_0}$ is attained and $\gamma=\omega_0\le \omega_{V_0} \le \gamma+\frac{\gamma}{\gamma+\gamma_0} V_0$.
\end{lemma}
\begin{proof}
 Since $V_0 \ge 0$, it is obvious that $\omega_{V_0} \ge \omega_0$. Also,
\begin{equation*}
  \<H\varphi_0,\varphi_0\> = \gamma + V_0
  \frac{\gamma}{\pi}\int_{\R^2} e^{-(\gamma+\gamma_0)|x|^2}
  dx= \gamma +\frac{\gamma}{\gamma+\gamma_0} V_0 ,
\end{equation*}
hence $\omega_{V_0} \le \gamma+\frac{\gamma}{\gamma+\gamma_0} V_0$. It remains to show that $\omega_{V_0}$ is attained. Let $(f_n)_{n\ge 1}$ be a minimizing sequence for $\omega_{V_0}$. It follows that $(f_n)_{n\ge 1}$ is a bounded sequence in $\Sigma$. Thus, there exists $\varphi \in \Sigma$ such that (up to a subsequence) $f_n \rightarrow \varphi$ weakly in $\Sigma$ and strongly in $L^r(\R^2)$ for all $2\le r<\infty$. We infer that
	\[
	\|\varphi\|_{L^2}^2 =\lim_{n\rightarrow \infty} \|f_n\|_{L^2}^2 =1.
	\]
	By the lower continuity of the weak convergence, we have
	\[
	\frac{1}{2} \|\nabla \varphi\|^2_{L^2} + \frac{\gamma^2}{2} \|x \varphi\|^2_{L^2} \le \liminf_{n\rightarrow \infty} \frac{1}{2} \|\nabla f_n\|^2_{L^2} + \frac{\gamma^2}{2} \|x f_n\|^2_{L^2}.
	\]
	We also have from the weak continuity of the potential energy (see e.g., \cite[Theorem 11.4]{LiLo01}) that
	\begin{align} \label{eq:weak-cont}
	\int_{\R^2} e^{-\gamma_0|x|^2} |\varphi(x)|^2 dx = \lim_{n\rightarrow \infty} \int_{\R^2} e^{-\gamma_0 |x|^2} |f_n(x)|^2 dx.
	\end{align}
	This implies
	\[
	\omega_{V_0} \le \frac{1}{2} \|\nabla \varphi\|^2_{L^2} + \int_{\R^2} V|\varphi|^2 dx \le \liminf_{n\rightarrow \infty} \frac{1}{2} \|\nabla f_n\|^2_{L^2} + \int_{\R^2} V|f_n|^2 dx= \omega_{V_0},
	\]
	hence $\varphi$ is a minimizer for $\omega_{V_0}$. The proof is complete.
\end{proof}

Consider now the Pohozaev identities related to \eqref{eq:ground-rad}.
\begin{lemma} \label{lem:poho-iden}
	Let $\phi \in \Sigma$ be a solution to \eqref{eq:ground-rad}. Then the following identities hold
	\begin{align}
	\frac{1}{2} \|\nabla \phi\|^2_{L^2} + \omega \|\phi\|^2_{L^2} &+ \int_{\R^2} V |\phi|^2 dx + \int_{\R^2} |\phi|^4 \ln (|\phi|^2)dx =0, \label{poho-1} \\
	\omega \|\phi\|^2_{L^2} + \int_{\R^2} V|\phi|^2 dx &+\frac{1}{2}\int_{\R^2} x \cdot \nabla V |\phi|^2dx + \frac{1}{2}\int_{\R^2} |\phi|^4 \ln \left(\frac{|\phi|^2}{\sqrt{e}}\right) dx =0, \label{poho-2} \\
	\frac{1}{2}\|\nabla \phi\|^2_{L^2} +\frac{1}{2}\|\phi\|^4_{L^4} &=\omega \|\phi\|^2_{L^2} + \int_{\R^2} V|\phi|^2 dx + \int_{\R^2} x \cdot \nabla V |\phi|^2 dx. \label{poho-3}
	\end{align}
	In particular, since $V_0\ge 0$, if $\omega \ge
       \frac{1}{2\sqrt{e}}+\tfrac{V_0}{e^2}$ or
        $\omega+\omega_{V_0}> \tfrac{1}{e}$, then $\phi \equiv  0$.
\end{lemma}

\begin{proof}
The identity \eqref{poho-1} follows by multiplying both sides of \eqref{eq:ground-rad} with $\bar{\phi}$ and taking the integration over $\R^2$. By multiplying both sides of \eqref{eq:ground-rad} with $x\cdot \nabla \bar{\phi}$, integrating over $\R^2$, taking the real part and using the fact that
\begin{align*}
\RE \(\int_{\R^2} \Delta \phi x \cdot \nabla \bar{\phi} dx\) &= 0, \\
\RE \( \int_{\R^2} V \phi x \cdot \nabla \bar{\phi} dx \) &= -\int_{\R^2} V|\phi|^2 dx  - \frac{1}{2} \int_{\R^2} x \cdot \nabla V |\phi|^2 dx, \\
\RE \(\int_{\R^2} \phi |\phi|^2 \ln(|\phi|^2) x \cdot \nabla \bar{\phi}dx\) &= -\frac{1}{2} \int_{\R^2} |\phi|^4 \ln \left(\frac{|\phi|^2}{\sqrt{e}}\right) dx,
\end{align*}
we get \eqref{poho-2}. Finally, \eqref{poho-3} follows directly from \eqref{poho-1} and \eqref{poho-2}.
\smallbreak

In view of the formula
\begin{equation*}
V(x)+\frac{1}{2}x\cdot \nabla V(x) = \gamma^2 |x|^2 + V_0\(1-\gamma_0 |x|^2\)e^{-\gamma_0 |x|^2},
\end{equation*}
and of the straightforward estimate
\begin{equation*}
\sup_{|x|^2>1/\gamma_0} \(\gamma_0 |x|^2-1\)e^{-\gamma_0  |x|^2}=\sup_{r>1} (r-1)e^{-r}=\frac{1}{e^2},
\end{equation*}
\eqref{poho-2}  yields
\begin{align*}
\omega \|\phi\|^2_{L^2} + \gamma^2\int_{\R^2} |x|^2|\phi|^2 dx &+ \frac{1}{2}\int_{|\phi|^2 >\sqrt{e}} |\phi|^4 \ln
\left(\frac{|\phi|^2}{\sqrt{e}}\right) \\
&\le  \frac{1}{2}\int_{|\phi|^2<\sqrt{e}}  |\phi|^4 \ln \left(\frac{\sqrt{e}}{|\phi|^2}\right)dx +\frac{V_0}{e^2}\|\phi\|_{L^2}^2.
\end{align*}
We also compute
\begin{align} \label{eq:fact}
\sup_{0<z<1} z \ln\left(\frac{1}{z}\right) =\frac{1}{e},
\end{align}
hence
\begin{align*}
\omega \|\phi\|^2_{L^2} + \gamma^2\int_{\R^2} |x|^2|\phi|^2 dx + \frac{1}{2}\int_{|\phi|^2 >\sqrt{e}} |\phi|^4 \ln\left(\frac{|\phi|^2}{\sqrt{e}}\right) \le\(\frac{1}{2\sqrt{e}} +\frac{V_0}{e^2}\)\|\phi\|^2_{L^2}.
\end{align*}
It follows that
\[
\gamma^2\int_{\R^2} |x|^2|\phi|^2 dx +  \frac{1}{2}\int_{|\phi|^2 >\sqrt{e}} |\phi|^4 \ln
\left(\frac{|\phi|^2}{\sqrt{e}}\right) \le \left(\frac{1}{2\sqrt{e}}+\frac{V_0}{e^2}-\omega\right)\|\phi\|^2_{L^2}.
\]
Since the left hand side is the sum of non-negative terms, if $\omega \ge \frac{1}{2\sqrt{e}}+\tfrac{V_0}{e^2}$, then we must have $\phi\equiv 0$.

Similarly, \eqref{poho-1} and \eqref{eq:omega0} yield
\begin{equation*}
\(\omega+\omega_{V_0}\)\|\phi\|_{L^2}^2 + \int_{\R^2}|\phi|^4\ln(|\phi|^2)dx\le 0,
\end{equation*}
hence
\begin{equation*}
\(\omega+\omega_{V_0}\)\|\phi\|_{L^2}^2 \le \int_{|\phi|^2<1}|\phi|^4\ln\(\frac{1}{|\phi|^2}\)dx\le \frac{1}{e}\|\phi\|_{L^2}^2,
\end{equation*}
where we have used \eqref{eq:fact} again. This shows that if $\omega + \omega_{V_0} > \frac{1}{e}$, then we must have $\phi \equiv 0$. The proof is complete.
\end{proof}

\begin{remark}
We infer from Lemma \ref{lem:poho-iden} that non-trivial solutions exist only if
\begin{align} \label{cond-omega-1}
\omega < \frac{1}{2\sqrt{e}}+\frac{V_0}{e^2} \quad \text{and} \quad \omega + \omega_{V_0} \le  \frac{1}{e}.
\end{align}
In particular, when $V_0=0$, the above conditions become
\begin{equation*}
\omega<\frac{1}{2\sqrt e}\quad \text{and} \quad \omega+\gamma\le \frac{1}{e}.
\end{equation*}
According to the value of $\gamma$, either the first condition or the second one is the more stringent.
\end{remark}

We now explain in more details why finding directly a solution to
\eqref{eq:ground-rad} (that is, not minimizing the energy under a mass
constraint, in which case $\omega$ corresponds to the - unknown -
Lagrange multiplier) seems intricate. In order to
  simplify the computations, we go back to the initial case \eqref{eq:V-intro0},
  that is, we assume $\gamma_0=\gamma$.
To find ground states related to \eqref{eq:ground-rad}, a standard way is to consider the following minimization problem
\begin{align} \label{eq:d-omega-rad}
m_\omega := \inf \left\{ S_\omega (f) \ : \ f \in \Sigma
\backslash \{0\},\  K_\omega(f)=0\right\},
\end{align}
and then show that minimizers for $m_\omega$ are indeed solutions to \eqref{eq:ground-rad}, where
\begin{align*}
S_\omega(f) &:= \frac{1}{2} \|\nabla f\|^2_{L^2} + \omega \|f\|^2_{L^2} + \int_{\R^2} V|f|^2 dx + \frac{1}{2} \int_{\R^2} |f|^4 \ln \left(\frac{|f|^2}{\sqrt{e}}\right) dx, \\
K_\omega(f) &:=  \|\nabla f\|^2_{L^2} + 2 \omega \|f\|^2_{L^2} + 2\int_{\R^2} V|f|^2 dx + 2\int_{\R^2} |f|^4 \ln(|f|^2) dx.
\end{align*}
Note that
\begin{align} \label{S-K-omega}
S_\omega(f) &= \frac{1}{2} K_\omega(f) -\frac{1}{2} \int_{\R^2} |f|^4 \ln\left(|f|^2\right) dx -\frac{1}{4} \|f\|^4_{L^4} \\
&=\frac{1}{4} K_\omega(f) +\frac{1}{4} \|\nabla f\|^2_{L^2} +\frac{\omega}{2} \|f\|^2_{L^2} +\frac{1}{2} \int_{\R^2} V|f|^2 dx -\frac{1}{4} \|f\|^4_{L^4}.
\end{align}

We have the following sufficient condition that ensures the minimizing problem \eqref{eq:d-omega-rad} is well-defined.
\begin{lemma} \label{lem-K-omega-nonempty}
	Let $\gamma>0$ and $V_0 \ge 0$. If
	\begin{align} \label{cond-omega-2}
	\omega < \frac{1}{2\sqrt{e}}- \gamma - \frac{V_0}{2},
	\end{align}
	then the set
	\[
	\left\{ f \in \Sigma \backslash \{0\} \ : \ K_\omega(f) =0 \right\}
	\]
	is not empty.
\end{lemma}

\begin{proof}
	Denote $f_b(x)= e^{-b|x|^2}$ with $b>0$. A direct computation shows
	\begin{align*}
	\|f_b\|^2_{L^2} = \frac{\pi}{2b}, \quad \|xf_b\|^2_{L^2} = \frac{\pi}{4b^2}, \quad & \|\nabla f_b\|^2_{L^2} = \pi, \quad \|f_b\|^4_{L^4} = \frac{\pi}{4b}, \\
	\int_{\R^2} V(x)|f_b(x)|^2 dx = \frac{\gamma^2\pi}{8b^2} + \frac{\pi V_0}{\gamma+2b}, \quad & \int_{\R^2} |f_b|^4 \ln (|f_b|^2)dx = -\frac{\pi}{8b}.
	\end{align*}
	For $\lambda>0$ to be chosen later, we have
	\begin{align*}
	K_\omega(\lambda f_b) &= \lambda^2 \|\nabla f_b\|^2_{L^2} + 2\omega \lambda^2 \|f_b\|^2_{L^2} + 2 \lambda^2 \int_{\R^2} V|f_b|^2 dx \\
	&\mathrel{\phantom{= \lambda^2 \|\nabla f_b\|^2_{L^2}}}+ 2\lambda^4 \ln(\lambda^2) \|f_b\|^4_{L^4} + \lambda^4 \int_{\R^2} |f_b|^4 \ln(|f_b|^2) dx \\
	&= \lambda^2 \pi \left( 1+\frac{\omega}{b}\right) + 2\lambda^2 \pi \left(\frac{\gamma^2}{8b^2} +\frac{V_0}{\gamma+2b}\right) + \frac{\lambda^4 \pi}{2b} \ln \left(\frac{\lambda^2}{\sqrt{e}}\right).
	\end{align*}
	It follows that
	\[
	\frac{K_\omega(\lambda f_b)}{\lambda^2 \pi} = 1+\frac{\omega}{b} + \frac{\gamma^2}{4b^2} + \frac{2V_0}{\gamma+2b} + \frac{\lambda^2}{2b} \ln \left(\frac{\lambda^2}{\sqrt{e}}\right).
	\]
	Using \eqref{eq:fact}, we see that the last term takes its maximal negative value at $\lambda^2=\frac{1}{\sqrt{e}}$. It follows that
	\begin{align*}
	\left. \frac{K_\omega(\lambda f_b)}{\lambda^2 \pi}\right|_{\lambda^2=\frac{1}{\sqrt{e}}} = \pi \left(1+\frac{\omega}{b} + \frac{\gamma^2}{4b^2} + \frac{2V_0}{\gamma+2b} - \frac{1}{2b\sqrt{e}}\right).
	\end{align*}
	Optimizing the sum of the first and third terms yields $b=\frac{\gamma}{2}$, hence
	\[
	\left.\frac{K_\omega(\lambda f_b)}{\lambda^2 \pi}\right|_{(\lambda^2, b)=\left(\frac{1}{\sqrt{e}}, \frac{\gamma}{2}\right)} = \pi \left( 2+\frac{2\omega}{\gamma} + \frac{V_0}{\gamma} - \frac{1}{\gamma\sqrt{e}}\right).
	\]
	This shows that if $\omega <\frac{1}{2\sqrt{e}} - \gamma - \frac{V_0}{2}$, then
	\[
	\left.K_\omega(\lambda f_b)\right|_{(\lambda^2, b)=\left(\frac{1}{\sqrt{e}}, \frac{\gamma}{2}\right)} <0.
	\]
	By the continuity argument, there exists $f \in \Sigma \backslash \{0\}$ such that $K_\omega(f) =0$.
\end{proof}
\begin{remark}
	The condition \eqref{cond-omega-2} is stronger than \eqref{cond-omega-1} since $V_0\ge 0$ and $\omega_{V_0} \le \gamma +\frac{V_0}{2}$.
\end{remark}
By the above remark, from now on, we consider the following condition on $\omega$
\begin{align} \label{cond-omega}
-\omega_{V_0} < \omega <\frac{1}{2\sqrt{e}}-\gamma-\frac{V_0}{2}.
\end{align}
This condition is equivalent to
\[
0<\omega+\omega_{V_0} < \frac{1}{2\sqrt{e}} +\omega_{V_0} -\gamma -\frac{V_0}{2}.
\]
To make the above inequality is not empty, we need
\[
\gamma+\frac{V_0}{2} -\omega_{V_0}<\frac{1}{2\sqrt{e}}.
\]
Note that the left hand side belongs to $\left(0,\frac{V_0}{2}\right)$ as $\gamma \le \omega_{V_0} \le \gamma+\frac{V_0}{2}$. So, if $V_0<\frac{1}{\sqrt{e}} \simeq 0.6$, then the above condition is satisfied. Note that only the values $V_0=0$ and $V_0=0.2$ are considered in \cite{TSKR19}.

To show the existence of minimizers for $m_\omega$, we need an estimate showing the boundedness in $\Sigma$ of minimizing sequences for $m_\omega$. However due to the interplay of various nonlinear terms, it is not clear to us how to prove this at the moment. Below we collect some estimates which may be helpful for future investigation.

\begin{remark}
	Let $\omega$ be as in \eqref{cond-omega}. Let $f\in \Sigma\backslash \{0\}$ be such that $K_\omega(f)=0$. Then there exist positive constants $C_1=C_1(\omega)$ and $C_2=C_2(\omega)$ such that
	\[
	\|f\|^2_{L^2} + \|f\|^4_{L^4} \le C_1(\omega) \|f\|^3_{L^3}, \quad \|\nabla f\|^2_{L^2} + \|f\|^4_{L^4} \le C_1(\omega) \|f\|^2_{L^2}, \quad \|f\|_{L^3}\ge C_2(\omega).
	\]	
	Indeed, since $K_\omega(f)=0$, we have
	\[
	\|\nabla f\|^2_{L^2} + 2\omega \|f\|^2_{L^2} + 2\int_{\R^2} V|f|^2 dx + 2\int_{\R^2} |f|^4 \ln \left(\frac{|f|^2}{\sqrt{e}}\right) dx +\|f\|^4_{L^4} =0
	\]
	hence
	\begin{align*}
	\|\nabla f\|^2_{L^2} + 2\omega \|f\|^2_{L^2} + 2\int_{\R^2} V|f|^2 dx + \|f\|^4_{L^4} &+ 2\int_{|f|^2>\sqrt{e}} |f|^4 \ln \left(\frac{|f|^2}{\sqrt{e}}\right) dx \\
	&= 2\int_{|f|^2<\sqrt{e}} |f|^4 \ln \left(\frac{\sqrt{e}}{|f|^2}\right) dx.
	\end{align*}
	By the definition of $\omega_{V_0}$, we see that
	\begin{align*}
	2(\omega+\omega_{V_0})\|f\|^2_{L^2} + \|f\|^4_{L^4} + 2\int_{|f|^2>\sqrt{e}} |f|^4 \ln \left(\frac{|f|^2}{\sqrt{e}}\right) dx \le C \|f\|^3_{L^3}
	\end{align*}
	which implies
	\[
	\|f\|^2_{L^2} + \|f\|^4_{L^4} \le C(\omega) \|f\|^3_{L^3}.
	\]
	Another estimate using \eqref{eq:fact} yields
	\begin{align}
	\|\nabla f\|^2_{L^2} + 2\int_{\R^2} V|f|^2 dx + \|f\|^4_{L^4} &+ 2\int_{|f|^2>\sqrt{e}}|f|^4 \ln \left(\frac{|f|^2}{\sqrt{e}}\right) dx \nonumber \\
	&\le 2\left(\frac{1}{\sqrt{e}}-\omega\right) \|f\|^2_{L^2} \label{eq:est-L4}
	\end{align}
	which implies
	\[
	\|\nabla f\|^2_{L^2} + \|f\|^4_{L^4} \le C(\omega) \|f\|^2_{L^2}.
	\]
	By the Gagliardo-Nirenberg inequality, we obtain
	\begin{align} \label{eq:est-L3}
	\|f\|^6_{L^3} \le C \|\nabla f\|^2_{L^2} \|f\|^4_{L^2} \le C(\omega) \|f\|^6_{L^2} \le C(\omega) \|f\|^9_{L^3},
	\end{align}
	hence $\|f\|^3_{L^3} \ge \frac{1}{C(\omega)}>0$.
\end{remark}

\begin{observation} \label{obs:1}
	We have the following useful estimates:
	\begin{align*}
	\|f\|^2_{L^2} &\le \|\nabla f\|_{L^2} \|xf\|_{L^2} \\
	&\le \frac{\sqrt{2}}{\gamma} \|\nabla f\|_{L^2} \left( \int_{\R^2} V|f|^2  dx\right)^{1/2} \nonumber\\
	&\le \frac{1}{\gamma}\(\frac{1}{2} \|\nabla f\|^2_{L^2} + \int_{\R^2} V|f|^2 dx \). \nonumber
	\end{align*}
	and
	\begin{align*}
	\|f\|^4_{L^4} \le \frac{1}{\|Q\|^2_{L^2}} \|\nabla f\|^2_{L^2} \|f\|^2_{L^2},
	\end{align*}
	where $Q$ is the unique positive radial solution to \eqref{eq:Q} satisfying \eqref{eq:estQ}.
\end{observation}

In the following remark, we point out another difficulty in finding ground states related to \eqref{eq:ground-rad}. More precisely, suppose that there is a minimizer $\phi$ for $m_\omega$, we are not able to show that $\phi$ is a solution to \eqref{eq:ground-rad}.

\begin{remark}
	Assume that $m_\omega$ is attained by a function $\phi$. Then there exists a Lagrange multiplier $\lambda \in \R$ such that $S'_\omega(\phi) =\lambda K'_\omega(\phi)$ or
	\begin{align} \label{eq:lagrange}
	-\Delta \phi + 2\omega \phi + 2V \phi &+ 2 \phi |\phi|^2 \ln(|\phi|^2) \\
	&= \lambda \(2\Delta \phi + 4 \omega \phi + 4 V \phi + 8 \phi |\phi|^2 \ln(|\phi|^2) + 4 \phi |\phi|^2 \). \nonumber
	\end{align}
	We want to show that $\lambda=0$ so that $\phi$ is a solution to \eqref{eq:ground-rad}. To see this, we first multiply both sides of \eqref{eq:lagrange} with $\bar{\phi}$ and integrate over $\R^2$ to get
	\[
	K_\omega(\phi) = \lambda \left(2\|\nabla \phi\|^2_{L^2} + 4\omega \|\phi\|^2_{L^2} + 4 \int_{\R^2} V|\phi|^2 dx + 8\int_{\R^2} |\phi|^4\ln(|\phi|^2)dx + 4 \|\phi\|^4_{L^4} \right).
	\]
	Since $K_\omega(\phi)=0$, we infer that
	\[
	\lambda \left( \int_{\R^2} |\phi|^4 \ln(|\phi|^2) dx + \|\phi\|^4_{L^4}\right) =0.
	\]
	Suppose that $\lambda \ne 0$, then we must have
	\begin{align} \label{eq:phi-1}
	\int_{\R^2} |\phi|^4 \ln(|\phi|^2) dx + \|\phi\|^4_{L^4} =0.
	\end{align}
	In particular, as $K_\omega(\phi)=0$, we have
	\begin{align} \label{eq:phi-2}
	\|\nabla\phi\|^2_{L^2} + 2\omega \|\phi\|^2_{L^2} + 2\int_{\R^2} V|\phi|^2 dx - 2 \|\phi\|^4_{L^4}=0.
	\end{align}
	Moreover, if we multiply both sides of \eqref{eq:lagrange} with $x \cdot \nabla \bar{\phi}$, integrate over $\R^2$, and take the real part, we get
	\begin{align*}
	2\omega &\|\phi\|^2_{L^2} + 2 \int_{\R^2} V|\phi|^2 dx + \int_{\R^2} x\cdot \nabla V |\phi|^2 dx + \int_{\R^2} |\phi|^4 \ln(|\phi|^2) dx - \frac{1}{2} \|\phi\|^4_{L^4} \\
	&=\lambda \(4\omega \|\phi\|^2_{L^2} + 4 \int_{\R^2} V|\phi|^2 dx + 2\int_{\R^2} x \cdot \nabla V |\phi|^2 dx + 4 \int_{\R^2} |\phi|^4 \ln(|\phi|^2) dx \)
	\end{align*}
	which implies
	\[
	(2\lambda-1) A + (4\lambda-1)B +\frac{1}{2} C=0,
	\]
	where
	\begin{align*}
	A:&= 2\omega \|\phi\|^2_{L^2} + 2 \int_{\R^2} V |\phi|^2 dx + \int_{\R^2} x\cdot \nabla V |\phi|^2 dx, \\
	B:&= \int_{\R^2} |\phi|^4 \ln(|\phi|^2)dx, \quad C:= \|\phi\|^4_{L^4}.
	\end{align*}
	From this and \eqref{eq:phi-1}, we have
	\[
	(2\lambda-1)A +\left(\frac{3}{2}-4\lambda\right)C=0.
	\]
However, it is not clear that $\lambda=0$, and it is not possible to conclude.
\end{remark}

\section{Characterization of prescribed mass minimizers}\label{sec:charac}
In this section, we will characterize the orbit of prescribed mass
standing waves obtained in Section \ref{sec:orbi}, in the absence of
the rotational effect, following the strategy from \cite{HaSt04}. Throughout this section, we assume that $\Omega = 0$ and $K_{3}= 0$.

Denote
\begin{align*}
\Sigma_{\R}:= \left\{ f \in H^1(\R^2,\R) \ : \ \int_{\R^2} |x|^2 f^2 dx <\infty\right\}.
\end{align*}
For $f \in \Sigma_{\R}$, we define
\[
E^{\R}_0(f):= \frac{1}{2} |\nabla f|^2_{L^2} + \int_{\R^2} V f^2 dx + \frac{1}{2} \int_{\R^2} f^4 \ln \left(\frac{f^2}{\sqrt{e}}\right) dx.
\]
Let $z =(f,g) \simeq f + i g$. We observe that
\[
z \in \Sigma \Longleftrightarrow f \in \Sigma_{\R}, \quad g \in \Sigma_{\R}.
\]
For $z \in \Sigma$, we consider
\[
E_0(z):= \frac{1}{2} \|\nabla z\|^2_{L^2} + \int_{\R^2} V|z|^2 dx +\frac{1}{2} \int_{\R^2} |z|^4 \ln \left(\frac{|z|^2}{\sqrt{e}}\right) dx.
\]
Here $|\cdot|_{L^2}$ and $\|\cdot\|_{L^2}$ denote the norms of $L^2(\R^2,\R)$ and $L^2(\R^2,\C)$, respectively.

For $c>0$, we consider the following minimizing problems:
\begin{align*}
I^{\R}_{0,c} :&=\inf \left\{ E^{\R}_0(f) \ : \ f \in S^{\R}_c\right\},\\
I_{0,c} :&= \inf \left\{ E_0(z) \ : \ z \in S_c\right\},
\end{align*}
where
\begin{align*}
S^{\R}_c:&= \left\{f \in \Sigma_{\R} \ : \ |f|^2_{L^2}=c\right\},\\
S_c :&= \left\{ z \in \Sigma \ : \ \|z\|^2_{L^2} = c\right\}.
\end{align*}
We also denote
\begin{align*}
W_c :&= \left\{ f \in S^{\R}_c \ : \ E^{\R}_c(f) = I^{\R}_{0,c},~ f >0 \right\},  \\
Z_c :&= \left\{ z \in S_c \ : \ E_0(z) = I_{0,c} \right\}.
\end{align*}

\begin{theorem}  \label{theo-charac}
	We have the following properties:
\begin{enumerate}
	\item For any $c>0, I_{0,c}^{\R}=  I_{0,c}$.
	\item If $z\in Z_{c}$, then  $|z|\in W_{c}$.
	\item If $z\in Z_{c}$ with $z=f+ig$, then
\begin{enumerate}
		\item $f\equiv 0$ or $f(x)\ne 0$ for all $x \in \R^2$.
		\item $g\equiv 0$ or $g(x)\ne 0$ for all $x \in \R^2$.
\end{enumerate}
	\item $Z_{c}=\left\{e^{i\sigma} \varphi \ : \ \sigma \in \R, \varphi \in W_c \right\}$.
\end{enumerate}
\end{theorem}

\begin{proof}
	The proof is essentially given in \cite{HaSt04}. For the reader's convenience, we provide some details.
	
(1) Let $z=(f,g)\in S_{c}$ be such that $E_{0}(z)= I_{0,c}$ and set $\varphi:=|z|=\sqrt{f^2+g^2}$. It follows that $\varphi \in S_{c}^{\R}$ and we have (see \cite[Proposition 2.2.]{HaSt04}) for $i=1,2$,
\[
\partial_{i}\varphi =\left\{
\begin{array}{c l}
\displaystyle  \frac{f \partial_i g + g\partial_i f}{\sqrt{f^2+g^2}} & \mbox{if } f^2+g^2>0, \\
0 & \text{otherwise.}
\end{array}
\right.
\]
We see that
\begin{align}
E_{0}(z)-E_{0}^{\R}(\varphi)&=\frac{1}{2} \left(\|\nabla z\|_{L^2}^2-|\nabla \varphi|_{L^2}^2\right) \nonumber\\
&=\frac{1}{2}\int_{\R^2} \left(|\nabla f|^2+|\nabla g|^2 -|\nabla \varphi|^2\right) dx \nonumber\\
&=\frac{1}{2} \int_{f^2+g^2>0}\sum_{i=1}^{2} \frac{\left(f\partial_i g-g\partial_i f\right)^2}{f^2+g^2} dx \nonumber\\
&\ge 0. \label{est-E0-z}
\end{align}
Thus
\[
I_{0,c} =E_{0}(z) \ge E_{0}^{\R}(\varphi)\ge I_{0,c}^{\R}\ge I_{0,c}
\]
which implies that
\begin{align} \label{I0c-varphi}
I_{0,c}^{\R} = I_{0,c}, \quad E^{\R}_0(\varphi) = I_{0,c}.
\end{align}

(2) If $z\in Z_c$, then $\varphi=|z|$ satisfies $\varphi \in S_{c}^{\mathbb{R}}$ and $E_{0}^{\mathbb{R}}(\varphi)=I_{0,c}$. Using regularity theory and maximum principle, we can deduce that  $\varphi\in C^{1}(\R^2)$ and $\varphi>0$. In particular, $|z| \in W_c$.

(3) We only prove (a) since the one for (b) is treated in a similar manner. Let $z=(f,g)\in Z_c$ and $\varphi=|z|$. By \eqref{est-E0-z} and \eqref{I0c-varphi}, we know that
\begin{align} \label{zero}
\int_{f^2+g^2>0} \frac{\left(f\partial_i g-g \partial_if\right)^2}{f^2+g^2} dx=0.
\end{align}
On the other hand, as $E_0(z)=I_{0,c}$, there exists $\mu\in \C$ such that for any $\xi \in \Sigma$,
\[
E_0'(z)(\xi)=\frac{\mu}{2}\int_{\R^{2}} z  \overline{\xi} + \xi  \overline{z}  dx.
\]
Putting $z=\xi$ and using regularity theory, we can deduce that $f$ and $g$ belong to $C^1(\R^2)$.

Now suppose that $f\equiv 0$. Denote
\[
\delta_f=\left\{x\in \R^2 \ : \ f(x)=0\right\}.
\]
The continuity of $f$ implies that $\delta_f$ is closed. Let us now prove that it is also an open set of $\R^2$. Pick a point $x_0\in \delta_f$. Since $\varphi(x_0)>0$, there exists a ball $B(x_0,\rho)$ centered at $x_0$ with radius $\rho>0$ such that $g(x)\ne 0$ for all $x\in B(x_0,\rho)$.
Observe that for each $x\in B(x_,\rho)$ and $i=1,2$,
\[
\frac{\left( f \partial_i g - g \partial_i f\right)^2}{f^2+g^2} = \left(\partial_i \left(\frac{f}{g}\right) \right)^2 \frac{g^4}{f^2+g^2}.
\]
From this and \eqref{zero}, we have
\[
\int_{B(x_0,\rho)}\left|\nabla \left(\frac{f}{g}\right)\right|^2 \frac{g^4}{f^2+g^2}dx=0.
\]
Hence $\nabla \left(\frac{f}{g}\right)=0$ on $B(x_0,\rho)$. This implies that $\frac{f}{g}=C$ on $B(x_0,\rho)$ for some constant $C>0$. As $f(x_0)=0$, we infer that $C=0$. Hence $f(x)=0$ for all $x \in B(x_0,\rho)$ or $B(x_0,\rho) \subset \delta_f$ which tells us that $\delta_f$ is an open set.

(4) Finally we prove that $Z_c=\left\{e^{i\sigma} \varphi \ : \ \sigma \in \R,~ \varphi \in W_c \right\}$. To this end, we consider two cases.

First case: If $g=0$, then $\varphi=|f|>0$ on $\R^2$. Thus $z=f=e^{i\sigma} \varphi$ with $\sigma=0$ if $f>0$, $\sigma=\pi$ if $f<0$.

Second case: If $g(x)\ne 0$ on $\R^2$, then for $i=1,2$, we have from \eqref{zero} that
\[
\int_{\R^2} \frac{\left(f\partial_i g-g\partial_i f\right)^2}{f^2+g^2} dx =
\int_{\R^2}  \partial_{i}\left(\frac{f}{g}\right)^2 \frac{f^4}{f^2+g^2} dx \equiv 0.
\]
It follows that $\nabla \left(\frac{f}{g}\right)=0$ on $\R^2$. This implies that $f=kg$ for $k \in \R$, hence $z=(k+i) g$ and
\[
\varphi=|z|=|k+i||g|.
\]
Let $\theta \in \R$ be such that $k+i = |k+i| e^{i\theta}$. We also write $g=|g|e^{i\vartheta}$ with $\vartheta= 0$ if $g>0$ and $\vartheta=\pi$ if $g<0$. It follows that
\[
z= (k+i)g = |k+i|e^{i\theta} |g| e^{i\vartheta} = \varphi e^{i\sigma}, \quad \sigma:= \theta + \vartheta \in \R.
\]
The proof is complete.
\end{proof}

\begin{remark}
  The statement of Theorem \ref{theo-charac} would remain
  true in the presence of the rotation and the line of attack would be
  identical provided that we had
	\begin{equation} \label{eq:est-A-z}
	\|\nabla_A z\|_{L^2} \ge |\nabla_A|z||_{L^2}
	\end{equation}
	and
	\begin{align*}
	\|\nabla_A z\|_{L^2}= |\nabla_A|z||_{L^2} \Longrightarrow \|\nabla z\|_{L^2}= |\nabla |z||_{L^2}.
	\end{align*}
However, we show that \eqref{eq:est-A-z} cannot hold in general.  By
definition, we have
\begin{align*}
  	\|\nabla_A z\|^2_{L^2} &= \sum_{j=1}^{2} \int_{\R^2}
                                 |(\partial_j - iA_j)z|^2 dx \\
  &= \sum_{j=1}^2 \int_{\R^2}\( |\partial_j z|^2 + iA_j \overline{z} \partial_j z - i A_j z \partial_j \overline{z} + A_j^2 |z|^2 \)dx.
\end{align*}
		Applying to $z=|z|$, we get
		\[
		|\nabla_A |z||^2_{L^2} = |\nabla |z| |^2_{L^2} + \int_{\R^2} |A|^2 |z|^2 dx.
		\]
		In our case $A=\Omega(-x_2,x_1)$, we have $|A|^2=\Omega^2 |x|^2$ and
		\[
		|\nabla_A|z||^2_{L^2} = |\nabla |z||^2_{L^2} + \Omega^2 \|xz\|^2_{L^2}.
		\]
		Let $\varphi(x) =e^{-|x|^2}$ and set $z_n(x):= e^{iA(y_n) \cdot x} \varphi(x+ y_n)$ with some $(y_n)_n \subset \R^2$ to be chosen later. Observe that
		\[
		|z_n(x)| = |\varphi(x+y_n)|, \quad \|\nabla_A z_n\|^2_{L^2} = |\nabla_A \varphi|^2_{L^2} = \pi \left(1+\frac{\Omega^2}{4}\right).
		\]
		On the other hand, by choosing $|y_n| \rightarrow \infty$ as $n\rightarrow \infty$, we have
		\begin{align*}
		|\nabla_A |z_n||^2_{L^2} &\ge \Omega^2 \int_{\R^2} |x|^2 |\varphi(x+y_n)|^2 dx = \Omega^2 \int_{\R^2} |x-y_n|^2 |\varphi(x)|^2 dx \\
	&	\ge \frac{\Omega^2}{4} |y_n|^2 \int_{|x| \le 1}
           |\varphi(x)|^2 dx \Tend n\infty \infty.
		\end{align*}
\end{remark}

\bibliographystyle{siam}

\bibliography{biblio}

\end{document}